\documentclass[11pt]{amsart}

\usepackage{amsmath,amsthm, amscd, amssymb, amsfonts, mathtools}
\usepackage[all]{xy}
\usepackage{verbatim}
\usepackage{color}
\usepackage{array}
\usepackage{multirow}
\usepackage{hhline}
\usepackage{enumitem}

\newtheorem{theorem}{Theorem}[section]
\newtheorem{lem}[theorem]{Lemma}
\newtheorem{cor}[theorem]{Corollary}

\newtheorem{pro}[theorem]{Proposition}

\theoremstyle{definition}
\newtheorem{fed}[theorem]{Definition}
\newtheorem{exa}[theorem]{Example}

\theoremstyle{remark}
\newtheorem{rem}[theorem]{Remark}

\newcommand{\ydh}{{}^{H}_{H}\mathcal{YD}}

\newcommand{\ydhr}{\mathcal{YD}_H^H}

\def\pf{\begin{proof}}
\def\epf{\end{proof}}

\newcommand{\ydgdual}{{}_{\ku^G}^{\ku^G}\mathcal{YD}}
\newcommand{\nc}{\newcommand}
\nc{\D}{\Delta}
\nc{\cP}{\mathcal{P}} \nc{\cU}{\mathcal{U}} \nc{\cX}{\mathcal{X}}
\nc{\cE}{\mathcal{E}} \nc{\cS}{\mathcal{S}} \nc{\cA}{\mathcal{A}}
\nc{\cC}{\mathcal{C}} \nc{\cO}{\mathcal{O}} \nc{\cQ}{\mathcal{Q}}
\nc{\cB}{\mathcal{B}} \nc{\cJ}{\mathcal{J}} \nc{\cI}{\mathcal{I}}
\nc{\cM}{\mathcal{M}} \nc{\cG}{\mathcal{G}} \nc{\cL}{\mathcal{L}}
\nc{\cK}{\mathcal{K}}

\renewcommand{\_}[1]{_{\left( #1 \right)}}

\nc{\e}{\varepsilon}
\def\bs{\boldsymbol}
\nc{\ba}{\mathbf{c}} \nc{\bb}{\ba'} \nc{\bt}{\mathbf{t}}
\nc{\bi}{\mathbf{i}} \nc{\bj}{\mathbf{j}}

\nc{\bB}{\mathbf{B}} \nc{\bS}{\mathbf{S}}  

\newcommand{\Sn}{{\mathbb S}}
\newcommand{\I}{{\mathbb I}}
\newcommand\gA{\mathfrak{A}}

\newcommand\GL{\operatorname{GL}}
\newcommand\id{\operatorname{id}}

\newcommand{\ydg}{{}^{\ku G}_{\ku G}\mathcal{YD}}

\newcommand\Hom{\operatorname{Hom}}
\newcommand\End{\operatorname{End}}

\newcommand\sgn{\operatorname{sgn}}

\newcommand\Aut{\operatorname{Aut}}

\newcommand\gr{\operatorname{gr}}

\newcommand\co{\operatorname{co}}
\newcommand\ad{\operatorname{ad}}

\newcommand\rg{\operatorname{rg}}

\newcommand\ord{\operatorname{ord}}

\newcommand\Alg{\operatorname{Alg}}
\newcommand\Cleft{\operatorname{Cleft}}

\newcommand\Rep{\operatorname{Rep}}
\newcommand\CoRep{\operatorname{Corep}}

\def\k{\Bbbk}
\def\ku{\Bbbk}

\def\homh{\Hom_H^H}

\def\ot{\otimes}

\def\s{\mathbb{S}}

\def\N{\mathbb{N}}

\def\P{\mathbb{P}}
\def\F{\mathfrak{F}}

\def\Bb{\mathbb{B}}

\def\B{\mathfrak{B}}

\def\cH{\mathcal{H}}
\def\mH{\mathcal{H}}
\def\eps{\varepsilon}

\def\mA{\mathcal{A}}
\def\mG{\mathcal{G}}
\def\mJ{\mathcal{J}}

\def\Ss{\mathcal{S}}
\def\mL{\mathcal{L}}

\def\mR{\mathcal{R}}
\def\mO{\mathcal{O}}
\def\mE{\mathcal{E}}

\def\mH{\mathcal{H}}

\def\lg{\langle}
\def\rg{\rangle}

\newcommand{\J}{{\mathcal J}}
\newcommand{\Gc}{{\mathcal G}}

\newcommand{\ydcd}{{}^{\ku^{\Sn_{4}}}_{\ku^{\Sn_{4}}}\mathcal{YD}}

\newcommand{\ydhd}{{}^{H^*}_{H^*}\mathcal{YD}}
\newcommand{\ydgg}{{}^{G}_{G}\mathcal{YD}}

\newcommand{\xij}[1]{x_{(#1)}}

\newcommand{\fij}[1]{f_{#1}}

\newcommand{\oK}{\overline{K}}
\newcommand{\ke}{\overline{e}}

\newcommand{\oG}{\overline{G}}

\begin{document}



\title[Copointed Hopf algebras over $\Sn_4$]{Copointed Hopf algebras over $\Sn_4$}

\author[Garc\'ia Iglesias, Vay]{Agust\'in Garc\'ia Iglesias and Cristian Vay}

\address{FaMAF-CIEM (CONICET), Universidad Nacional de C\'ordoba,
Medina A\-llen\-de s/n, Ciudad Universitaria, 5000 C\' ordoba, Rep\'ublica Argentina.}

\email{(aigarcia|vay)@famaf.unc.edu.ar}

\thanks{The work was partially supported by CONICET,
FONCyT PICT 2015-2854 and 2016-3957, Secyt (UNC), the MathAmSud project
GR2HOPF}

\begin{abstract}
We study the realizations of certain braided vector spaces of rack type 
as Yetter-Drinfeld modules over a cosemisimple Hopf algebra $H$. We apply the strategy developed in \cite{AAGMV} to compute their liftings and 
use these results to obtain the classification of finite-dimensional copointed Hopf algebras over $\Sn_4$. 
\end{abstract}

\maketitle

\section{Introduction}\label{sec:intro}

A braided vector space is a pair $(V,c)$ where $V$ is a vector space and $c\in \GL(V\ot V)$ is a map satisfying the {\it braid equation}
\[
(c\ot \id)(\id\ot c)(c\ot \id)=(\id\ot c)(c\ot \id)(\id\ot c).
\] 
If $H$ is a Hopf algebra, a {\it realization of $V$ over $H$} is an structure of Yetter-Drinfeld $H$-module on the vector space $V$ in such a way that the braiding $c$ 
coincides with the categorical braiding of $V$ as an object in $\ydh$.  

A {\it lifting} of $V\in\ydh$ is a Hopf algebra $A$ such that $\gr A$ is isomorphic to $\B(V)\# H$, that is the bosonization of the Nichols algebra of $V$ with $H$. In particular $H\simeq A_{(0)}$, the coradical of A, and $(V,c)$ is said to be the {\it infinitesimal braiding of $A$}. 

In \cite{AAGMV} we developed a strategy  to compute the liftings of a given $V\in \ydh$ as cocycle deformations of $\B(V)\# H$. 
In a few words, the strategy produces a family of $H$-module algebras $\mE(\bs\lambda)$, obtained as deformations of $\B(V)$. Set $\mH=\B(V)\# H$. If $\mE(\bs\lambda)\neq 0$, then $\mA(\bs\lambda)=\mE(\bs\lambda)\# H$ is an $\mH$-cleft object and the associated Schauenburg's left Hopf algebra $\mL(\bs\lambda)=L(\mA(\bs\lambda),\mH)$ is a lifting of $V\in\ydh$.

In particular, the algebras $\mE(\bs\lambda)$ are deformations of the Nichols algebra itself, and do not depend a priori of the realization, in the sense that they can be defined for generic parameters. The choice of a realization thus brings a restriction on these parameters as a second step.

Hence this new approach reduces the lifting problem to checking that certain algebras are nonzero. This technical step is solved for the cases we study here by means of computer program \cite{GAP} and the package \cite{GBNP}. 

In the present article, we follow this strategy to investigate the quadratic deformations of a Nichols algebra $\B(V)$ where $V$ is a braided vector space of rack type $V(X,q)$ or $W(q,X)$, cf.~\S\ref{sec:racks-bvs}. We give a necessary condition to realize such a $V$ over a Hopf algebra and we explore how the quadratic relations of $\B(V)$ are deformed. As a byproduct, we deduce the quadratic relations of $\B(W(X,q))$, using \cite{GV} and the corresponding 
description of $\B(V(X,q))$ given in \cite{GG}.

This general framework allows us to obtain new classification results about pointed and copointed Hopf algebras. We recall that a Hopf algebra $A$ is said to be {\it pointed} if $A_{(0)}=\k G(A)$ and {\it copointed} when $A_{(0)}=\k^G$ for some non-abelian group $G$. Our main result is the following.

\begin{theorem}\label{thm-copointed}
Let $L$ be a finite-dimensional copointed Hopf algebra over $\k^{\Sn_4}$, $L\not\simeq \k^{\Sn_4}$. Then $L$ is isomorphic to 
one and only one of the algebras in the following list:
\begin{enumerate}\renewcommand{\theenumi}{\alph{enumi}}
\renewcommand{\labelenumi}{(\theenumi)}
\item  $\cH_{[\ba]}$, $\ba \in \gA$, cf.~Definition  \ref{def: Aa}.
\item  $\cH^\chi_{[\ba]}$, $\ba \in \gA$, cf.~Definition  \ref{def: Aachi}.
\item  $\widetilde\cH_{[\ba]}$, $\ba \in \widetilde\gA$, cf.~Definition  \ref{def: Ca}.
\end{enumerate}
In particular, $L$ is a cocycle deformation of $\gr L$.
\end{theorem}
\noindent
Finite-dimensional copointed Hopf algebras over $\k^{\Sn_3}$ are classified in \cite{AV}.
\pf
Finite-dimensional Nichols algebras over $\Sn_4$ are classified in \cite[Theorem 4.7]{AHS} and every such $L$ is generated in degree one by \cite[Theorem 2.1]{AG1}. The algebras listed in the theorem are a complete family of deformations of these Nichols algebras by Propositions 
\ref{prop:liftings of On2} and \ref{prop:liftings of Y}. The liftings are constructed using the strategy in \cite{AAGMV}, so they arise as cocycle deformations of their graded versions, see 
Propositions \ref{pro:copointed-review}.
\epf

\subsection{Pointed Hopf algebras}\label{sec:list} We fix the following list of pairs $(X,q)$ of a rack $X$ and a 2-cocycle $q\in Z^2(X,\k)$, see \S \ref{sec:examples} for unexplained notation:
\begin{enumerate} 
\item[(i)] The conjugacy class of transpositions $\cO_2^4\subset \Sn_4$, $q\equiv -1$;
\item[(ii)] The conjugacy class of transpositions $\cO_2^4\subset \Sn_4$, $q=\chi$;
\item[(iii)] The conjugacy class of 4-cycles $\cO_4^4\subset \Sn_4$, $q\equiv -1$.
\end{enumerate} 

We turn our attention to pointed Hopf algebras and extend some classification results about pointed Hopf algebras over $\Sn_4$ to any group with a realization of the right braided vector space. Some of these algebras have been considered previously in the literature, although not with this generality. This is the content of the next theorem.

See \eqref{eqn:H-pointed} for the presentation of the Hopf algebras $\mH(\bs\lambda)$ associated to each family of parameters $\bs\lambda\in\bs\Lambda(X,q)$.

\begin{theorem}\label{thm-pointed}
Let $H$ be a cosemisimple Hopf algebra and let $(X,q)$ be as in \S \ref{sec:list}.
Let $L$ be a Hopf algebra whose infinitesimal braiding $M$ is a principal realization of $V=V(X,q)$ in $\ydh$. Then there is $\bs\lambda\in\lambda(X,q)$ such that $L\simeq \mH(\bs\lambda)$.
In particular, $L$ is a cocycle deformation of $\gr L\simeq \B(M)\# H$.
\end{theorem}
\pf
By Proposition \ref{pro:pointed-review}, $\mH(\bs\lambda)$ is a lifting of $V$ for each $\bs\lambda\in\Lambda$.

On the other hand, any such $L$ is a lifting of $M$ by \cite[Theorem 2.1]{AG1} and references therein. 
By Corollary \ref{cor:condition}, there is $\bs\lambda$ such that $L\simeq \mH(\bs\lambda)$, thus $L$ is a cocycle deformation of 
$\B(M)\# H$ by Proposition \ref{pro:pointed-review}. 
\epf

When $G=\Sn_4$, Theorem \ref{thm-pointed} is \cite[Main Theorem]{GG}, where the classification is first completed, using \cite[Theorem 4.7]{AHS}. The case $(\mO_2^4,-1)$ had been fully understood previously in \cite[Theorem 3.8]{AG1}. Theorem \ref{thm-pointed} also gives an alternative proof to \cite[Corollary 7.8]{GIM}, see also \cite{GaM}, where it is shown that these algebras are cocycle deformations of their graded versions.

\medskip

The article is organized as follows. 
In \S \ref{sec:pre} we collect some preliminaries on Hopf algebras and racks and in \S \ref{sec:liftings} we recall the strategy developed in \cite{AAGMV} to compute liftings, we review the basic steps in the context of copointed Hopf algebras. 
In \S \ref{sec:realizations} we study realization of braided vector spaces associated to racks and cocycles. 
In \S \ref{sec:pointed} we turn to pointed Hopf algebras and obtain new classification results, summarized in Theorem \ref{thm-pointed}. 
In \S \ref{sec:copointed} we present our main result Theorem \ref{thm-copointed}, which is heavily inspired by \cite{AV2} and follows by our work in \cite{AAGMV}. The article contains an Appendix in which we define three families of algebras and show that they are non-trivial, using \cite{GAP}. The program files together with the log files are hosted on the authors' personal webpages, see \texttt{www.famaf.unc.edu.ar/$\sim$(aigarcia|vay)}.

\subsection*{Acknowledgments.} Parts of this work were completed while the first author was visiting Ben Elias in the University of Oregon (USA) and the second was visiting Simon Riche in the University of Clermont Ferrand (France); both visits supported by CONICET. The authors warmly thank these colleagues for their hospitality. We thank the referee for his/her suggestions.

\section{Preliminaries}\label{sec:pre} 

We shall work over an algebraically closed field $\k$ of characteristic zero. We write $\P^k$ for the projective space associated to $\k^{k+1}$ and $[v]$ for the class of $0\neq v\in \k^{k+1}$.
 If $A$ is a $\k$-algebra and $X\subset A$,
then $\langle X\rangle$ is the two-sided ideal generated by $X$. We denote by $\Alg(A,\k)$ the set of algebra maps $A\to \k$.

Let $G$ be a finite group. We denote by $\ku G$ its
group algebra and by $\ku^G$ the function algebra on $G$. The usual basis of $\ku G$ is denoted by 
$\{g:g\in G\}$, so $\{\delta_g:g\in G\}$ is its dual basis in $\ku^G$. We set $e$ the identity element of $G$.

Fix $n\in\N$; we set $\I_n\coloneqq\{1, \dots,n\}\subset \N$. The symmetric group in $n$ letters 
is
denoted by $\Sn_n$ and $\Bb_n$ shall denote the {\it braid group} in $n$ strands. 
We let $\sgn:\Sn_n\to\k$ denote the sign representation of $\Sn_n$. If $X$ is a finite set, we write $|X|$ for the cardinal of $X$.

\subsection{Hopf algebras}
Let $H$ be a Hopf algebra; 
we write by $m:H\ot H\to H$, resp. $\Delta:H\to H\ot H$, for the multiplication, resp. comultiplication, of $H$. We write $\ast$ for the convolution product in the algebra $H^*$.

We denote by
$\{H_{(i)}\}_{i\geq0}$ the coradical filtration of $H$ and by $\gr H=\oplus_{n\geq0}\gr^n
H=\bigoplus_{n\geq0}H_{(n)}/H_{(n-1)}$ the associated graded coalgebra of $H$; $H_{(-1)}=0$. This is a graded Hopf algebra when 
$H_{(0)}\subset H$ is a subalgebra. We write $G(H)\subset H$ for the group of group-like elements of $H$; in particular $G(H)\subset H_{(0)}$. We write $P_{g,g'}(H)\subset H$ for the set of $(g,g')$-skew primitive elements in $H$, $g,g'\in G(H)$, and set $P(H)=P_{1,1}(H)$.

We denote by $\Rep H$, resp. $\CoRep H$, the tensor category of $H$-modules, resp. $H$-comodules. 
An $H$-(co)module algebra is thus an algebra in $\Rep H$, resp. $\CoRep H$. Recall that when $\mA\in\Rep H$ is an algebra, then $A=\mA\ot H$ becomes a $\k$-algebra, denoted $\mA\# H$, with multiplication
\begin{align}\label{eqn:smash}
(a\ot h)(a'\ot h')=a(h_{(1)}\cdot a')\ot h_{(2)}h'.
\end{align}

\subsubsection{Cleft objects}\label{sec:cleft}
An $H$-comodule algebra $A$ is said to be a (right) {\it cleft object} of $H$ if it has trivial coinvariants and there is a convolution-invertible comodule isomorphism $\gamma:H\to A$. If we choose $\gamma$ so that $\gamma(1)=1$, then we say it is a {\it section}. 
In this setting, there is a new Hopf algebra $L=L(A,H)$, together with an algebra coaction $A\to L\ot A$ such that $A$ becomes a $(L,H)$-bicleft object. Moreover, $L$ is a cocycle deformation of $H$ and any cocycle deformation arises in this way; see \cite{S} for details.

\subsection{Nichols algebras}
Let $(V,c)$ be a braided vector space. We denote by $\B(V)$ the {\it Nichols algebra} of $V$. Recall that this is the quotient  $T(V)/\cJ(V)$, where $\cJ(V)=\bigoplus_{n\geq 2}\cJ^n(V)$ and each homogeneous component $\cJ^n(V)$ is the kernel of the so-called $n$th quantum symmetrizer $\varsigma_n\in\End(V^{\ot n})$. See \cite{AS2} for details. 
We write $\cJ_r(V)\subset T(V)$ for the ideal generated by  $\bigoplus_{2\leq n\leq r}\cJ^n(V)$ and denote by $\widehat{\B}_r(V)$ the $r$th-approximation of $\B(V)$. Notice that $\cJ^2(V)=\cJ_2(V)$.

\subsection{Yetter-Drinfeld modules}

We write $\ydh$ for the category of Yetter-Drinfeld modules over $H$; this is a braided tensor category.
We denote by $\homh(V,W)$ the space of morphisms $V\to W$ in $\ydh$.
 We recall from  \cite[Proposition 2.2.1]{AG1} that there are braided equivalences $\ydh\simeq \ydhr$ when $H$ has bijective antipode and $\ydh\simeq \ydhd$ when $H$ is finite-dimensional. When $G$ is a finite group, the equivalence $\ydg\simeq \ydgdual$ will be of special interest in our setting.
We shall also write $\ydgg\coloneqq\ydg$.

\subsubsection{Liftings}If $A$ is a lifting of $V\in\ydh$, that is $\gr A\simeq \B(V)\# H$ cf.~\S \ref{sec:intro}, then there is an epimorphism of Hopf
algebras $\phi:T(V)\#H\to A$, the so-called {\it lifting map}  \cite[Proposition 2.4]{AV},  such that
\begin{align}\label{eq:properties of A and phi}
\phi_{|H}=\id,&& \phi_{|V\#H}\mbox{ is injective}&&\mbox{ and }&&\phi((\ku\oplus
V)\#H)=A_{(1)}.
\end{align}

Let $(V,c)$ be a braided vector space with a realization $V\in\ydh$. A lifting of $(V,c)$ over $H$ is a lifting of $V\in\ydh$. The realization $V\in \ydh$ is said to be {\it principal} when there is a basis $\{v_i\}_{i\in I}$ of $V$ and elements $\{g_i\}_{i\in I}\in G(H)$ such that the $H$-coaction on $V$ is determined by $v_i\mapsto g_i\ot v_i$, $i\in I$. 

\subsection{Racks}\label{sec:racks}
We recall that a {\it rack} $X=(X,\rhd)$ is a pair consisting of a nonempty set $X$ and an operation $\rhd:X\times X\rightarrow
X$ satisfying a self distributive law:
\[
x\rhd (y\rhd z)=(x\rhd y)\rhd (x\rhd z), \quad x,y,z\in X
\] 
and such that the maps $\phi_x\colon X\longmapsto X$, $y\mapsto x\rhd y$, $y\in X$,
are bijective for each $x\in X$. When $\phi_x=\phi_y$ implies $x=y$ in $X$, the rack is said to be {\it faithful}.  A rack is called {\it indecomposable} if it cannot be written as a disjoint union $X=Y\sqcup Z$ of two subracks $Y,Z\subset X$. A quandle is a rack in which $x\rhd x=x$, $x\in X$.

The prototypical example of a rack is given by $X=\mO\subset G$ a conjugacy class inside a group $G$, with $g \rhd h=ghg^{-1}$, $g,h\in \mO$; notice that this is indeed a quandle. 

The enveloping group $G_X$ of $X$ is the quotient of the free group $F(X)=\langle f_x\mid x\in X\rangle$ by the relations $f_xf_y=f_{x\rhd y}f_x$ for all $x, y\in X$. This is an infinite group. The finite 
enveloping group $F_X=G_X/S_X$ is defined as the quotient of $G_X$ by the normal subgroup $S_X=\lg f_x^{n_x}, x\in X\rg$; here $n_x=\ord\phi_x$, $x\in X$.

A {\it 2-cocycle} on $X$ is a function  $q:X\times X\rightarrow\ku^\times$, $(x,y)\mapsto q_{x,y}$, satisfying 
\[q_{x,y\rhd z}q_{y,z}=q_{x\rhd y,x\rhd z}q_{x,z}, \qquad \text{all }x,y,z\in X.\]

We write $Z^2(X,q)$ for the set of 2-cocycles on $X$. We say that a 2-cocycle is {\it constant} if $q_{x,y}=\omega$, $\forall\,x,y\in X$ and a fixed $\omega\in\k^\times$; we write 
$q\equiv\omega$.

\subsubsection{Braided vector spaces associated to racks}\label{sec:racks-bvs}

If $(X,\rhd)$ is a rack and $q$ is a 2-cocycle on $X$, then a structure of braided vector space on the linear span of $\{v_x|x\in X\}$ is determined by
\begin{align}\label{eqn:rack-br}
c(v_x\ot v_y)=q_{x,y}v_{x\rhd y}\ot v_x, \qquad x,y\in X.
\end{align} 
We denote this braided vector space by $V(X,q)$ and refer to a realization of $V(X,q)$ over $H$ as a realization of $(X,q)$. We write $\B(X,q)\coloneqq\B(V(X,q))$ for the corresponding 
Nichols algebra.

There is another braided vector space associated to $(X,q)$ which we denote $W(q,X)$. Following \cite{GV}, this is the vector space spanned by $\{w_x|x\in X\}$ with braiding
\begin{align}\label{eqn:rack-br-dual}
c(w_x\ot w_y)=q_{y,x}w_y\ot w_{y\rhd x}, \qquad x,y\in X.
\end{align} 
We set $\B(q,X)\coloneqq\B(W(q,X))$, see \S\ref{subsec:WqX}.

\subsection{Quadratic Nichols algebras}\label{subsec:quadratic}

A description of the 2nd (or quadratic) approximation $\widehat{\B}_2(X,q)$ can be found in \cite[Lemma 2.2]{GG}. The quadratic relations in $\B(X,q)$ are para\-metrized by a given subset $\mR'=\mR'(X,q)$ of equivalence classes in $\mR=X\times X/\sim$, where $\sim$ stands for the relation generated by $(i,j)\sim(i\rhd j,i)$. 

More precisely, if $(i,j)\in \mR$, then it defines a class $C\in\mR$ as 
\[
C=\{(i_2,i_1), \ldots, (i_{|C|},i_{|C|-1}) , (i_1,i_{|C|})\},
\] with $i_1=j$, $i_2=i$, $i_{h+2}=i_{h+1}\rhd i_h$, $1\leq h\leq |C|-2$ and $i_1=i_{|C|}\rhd i_{|C|-1}$. Thus, $\mR'$ consists of those classes for which $\prod_{h=1}^{|C|} q_{i_{h+1},i_h} =(-1)^{|C|}$. Set
\begin{align}\label{eqn:qrels}
b_C:&= \sum_{h=1}^{|C|}\eta_h(C)
\, v_{i_{h+1}}v_{i_h}, \quad C\in \mR'(X,q);
\end{align}
where $\eta_1(C)=1$ and $\eta_h(C)=(-1)^{h+1}q_{{i_2i_1}}q_{{i_3i_2}}\ldots q_{{i_hi_{h-1}}}$, $h\ge 2$. Hence
\[\mJ^2(X,q)=\ku\{ b_C:C\in \mR'(X,q) \}.\] See {\it loc.~cit.} for unexplained notation and details. 

We shall write $C[h]=\{(i_{h+1},i_h),\dots,(i_{2},i_1),\dots,(i_{h},i_{h-1})\}$, so $C=C[1]$. Also, $x\rhd(i,j)\coloneqq(x\rhd i,x\rhd j)$ and $x\rhd C\coloneqq\{x\rhd (i_{2},i_1),\ldots, x\rhd (i_1,i_{|C|}) \}$.

\

We shall also give a description of the set $\mJ^2(q,X)$ in \S \ref{sec:quadratic-q-X} below.

\begin{rem}Let $X$ be a quandle.
Assume that $\B(X,q)$ is quadratic and finite-dimensional. Then $q_{x,x}=-1$ for all $x\in X$.

\smallskip

Indeed, $q_{x,x}$ is a root of 1, different from 1, for each $x\in X$, as otherwise $\k [v_x]\subseteq \B(X,q)$ and we assume that $\dim\B(X,q)<\infty$. Now, if $N_x=\ord(q_{x,x})$, then $v_x^{N_x}\in \J(X,q)$; hence $N_x=2$. \qed
\end{rem}

\subsection{Nichols algebras associated to symmetric groups}\label{sec:examples}
Let us fix $n\geq 3$ and let
$X=X_n=\cO_2^n$ be the rack of the conjugacy class of transpositions in $\Sn_n$ and let $Y=\cO_4^4$ be the rack given by the conjugacy class of $\mbox{\footnotesize (1234)}\in\Sn_4$. 

We shall consider the constant cocycle $q\equiv-1$ on $X_n$ and $Y$. Also, let $\chi:\Sn_n\times\cO_2^n\rightarrow\ku^\times$ be the map defined in \cite{MS} by 
\begin{align}\label{eqn:MS}
\chi(g,(ij))=
\begin{cases}
1 & \mbox{ if }g(i)<g(j),\\
-1& \mbox{ if }g(i)>g(j),
\end{cases} \quad g\in\Sn_n.
\end{align}
Then $\chi\coloneqq\chi_{|X\times X}\colon X\times X\to\k^\times$ is a non-constant 2-cocycle on $X$. 

\begin{theorem}\label{thm:nichols}
Let $n=3,4,5$.
\begin{enumerate}\renewcommand{\theenumi}{\alph{enumi}}
\renewcommand{\labelenumi}{(\theenumi)}
\item\cite[\S 6.4]{MS},\cite{grania} The ideal $\cJ(X_n,-1)$  is generated by $\xij{ij}^2$,
\begin{align}
\label{eq:Vn}\xij{ij}\xij{kl}+\xij{kl}\xij{ij}&&\mbox{and}&&\xij{ij}\xij{ik}+\xij{ik}\xij{jk}+\xij{jk}\xij{ij}
\end{align}
for all $(ij),(kl),(ik)\in\cO_2^n$ with $\#\{i,j,k,l\} = 4$. 

\smallskip
\item\cite[\S 6.4]{MS},\cite{grania} The ideal $\cJ(X_n,\chi)$  is generated by $\xij{ij}^2$ and 
\begin{align}\label{eq:Wn}
\begin{split}
&\xij{ij}\xij{kl}-\xij{kl}\xij{ij}\,\mbox{ for all }\,\#\{i,j,k,l\} = 4,\\
&\xij{ij}\xij{ik}-\xij{ik}\xij{jk}-\xij{jk}\xij{ij},\,\xij{ik}\xij{ij}-\xij{jk}\xij{ik}-\xij{ij}\xij{jk}
\end{split}
\end{align}
for all $1\leq i<j<k\leq n$. 

\smallskip
\item\cite[Theorem 6.12]{AG3}
The ideal $\cJ(Y,-1)$ is generated by $x_\sigma x_{\sigma^{-1}}+x_{\sigma^{-1}}x_\sigma$,
\begin{align}\label{eq:U}
x_\sigma^2\quad\mbox{and}\quad x_\sigma x_{\tau}+x_\nu x_\sigma+x_\tau x_\nu
\end{align}
for all $\sigma,\tau\in\cO_4^4$ with $\sigma\neq\tau^{\pm1}$ and $\nu=\sigma\tau\sigma^{-1}$.
\end{enumerate}
\end{theorem}

Theorem \ref{thm:nichols} also applies if we consider the braided vector spaces $W(q,X)$ associated to these conjugacy classes, see Corollary \ref{cor:nichols}.

\begin{rem}
The definitions of the algebras in items (1) and (2) of Theorem \ref{thm:nichols} make sense for any $n\geq 6$. They are known as Fomin-Kirillov algebras; were introduced in \cite{FK}. It is an open problem to state if they are examples of Nichols algebras and whether they are finite-dimensional or not.
\end{rem}

\section{Lifting via cocycle deformation}\label{sec:liftings}\label{sec:strategy}

Let $H$ be a cosemisimple Hopf algebra and $V\in\ydh$. Let $\{x_1,\dots,x_\theta\}$ be a basis of $V$ and set $\I=\I_\theta$. 
We set $\B(V)=T(V)/\cJ(V)$ and assume that the ideal $\cJ(V)$ is finitely generated. 

We recall the method developed in \cite{AAGMV} to compute the liftings of $V$. We fix a minimal set $\mG$ of homogeneous generators of $\cJ(V)$ and consider an {\it adapted stratification} $\mG =  \mG_0 \cup \mG_1 \cup\dots \cup \mG_{N}$. For $0\le k \le N$, we set
\begin{align*}
\B_0 &:= T(V), &  \cH_0 &= T(V)\#H,
\\
\B_k &:= \B_{k-1} / \langle\Gc_{k-1}\rangle, &
\cH_k &=
\B_k\#H.
\end{align*}

By definition (the image of) $\mG_k$ is a basis of a Yetter-Drinfeld submodule of $P(\B_k)$. Then $\B_k$ is a braided Hopf algebra in $\ydh$ and $\cH_k$ a Hopf algebra. See \cite[5.1]{AAGMV} for details. We set $Y_k\coloneqq\ku\langle\Ss(\mG_k)\rangle\subset \mH_k$.

Let $V=\k\{y_i\}_{i\in \I}$ be another copy of $V$ and let $\cA_0$ denote the algebra $T(V)\#H$. Then $\cA_0$ is a cleft object of $\cH_0$ with coaction induced by $\Delta$ and section $\gamma_0=\id$.
Moreover, $\cL_0:=L(\cA_0,\cH_0)\simeq\cH_0$. Then $\cA_0$ is a $(\cH_0,\cL_0)$-bicleft object; the coaction of $\cL_0$ on $\cA_0$ is $\Delta$.

\smallskip

We fix the singleton $\Lambda_0=\{\mA_0\}\subset \Cleft(\mH_0)$ and we proceed recursively, for each $0\leq k\leq N$, following the steps in \cite[5.2]{AAGMV}. Namely,

\

\noindent
\texttt{Step 1.} We construct a family $\Lambda_{k+1}\subset \Cleft\mH_{k+1}$ \cite[5.2 (1b)]{AAGMV}, starting with $\Lambda_k$.
More precisely, we collect in $\Lambda_{k+1}$ all nonzero quotients 
\begin{align*}
\cA_{k+1}=\cA_{k}/\langle\varphi(Y_k^+)\rangle, \quad \varphi\in\Alg^{\cH_k}(Y_k,\cA_{k}),
\end{align*}
see \cite[Theorem 3.3]{AAGMV}. The coaction is induced by $\Delta_{|V\# H}$ and the section $\gamma_k$ satisfies $\gamma_{k|H}=\id$ \cite[Proposition 5.8 (b)]{AAGMV}.
\smallskip

\

\noindent
\texttt{Step 2.} We compute $L(\mA_{k+1},\mH_{k+1})$ for all $\mA_{k+1}\in\Lambda_{k+1}$ \cite[5.2 (2)]{AAGMV}. 

\

\noindent
\texttt{Step 3.} We check that any lifting of $V$ can be obtained as $L(\mA_{N+1},\mH_{N+1})$, form some $\mA_{N+1}\in\Lambda_{N+1}$  \cite[5.2 (3)]{AAGMV}.

\subsection{On step 1}\label{sec:on step 1}
We further review the algebras $\mA_k$ and the recursion in the first step.
To do this, we pick $k\geq 0$ and fix the following setting:
\begin{itemize}[leftmargin=*]
\item $\mA_k\in\Cleft \mH_k$, with section $\gamma_k:\cH_k\rightarrow\cA_k$.
\item $\Gc_k=\{u_i\}_{i\in\I_n}$ with associated comatrix elements $\{E_{ij}\}_{i,j}$ i.e.~the $H$-coaction on $\Gc_k$ is determined by \begin{align}\label{eqn:comatrix-relations}
u_i\mapsto \sum_j E_{ij}\ot u_j, \qquad i\in\I_n.
\end{align}
\item $Y_k$ is the subalgebra of $\mH_k$ generated by $\Ss(\Gc_k)$.
\item $\varphi_{\bs\lambda}:\ku\Ss(\Gc_k)\longrightarrow\mA_k$, for $\bs\lambda=(\lambda_i)_{i=1}^n\in\ku^n$, is the $\mH_k$-comodule map defined  by
\begin{align}\label{eqn:the form of varphi}
\varphi_{\bs\lambda}(\Ss(u_i))&=\gamma_k(\Ss(u_i))+\sum_{j=1}^n\lambda_j\, \Ss(E_{ij}),\qquad i\in\I_n.
\end{align}

\end{itemize}

By construction, there are projection and inclusion maps $\xymatrix{\mA_k \ar@<0.4ex>@{->}[r]^{\pi}
& H \ar@<0.4ex>@{->}[l]^{\iota=\gamma_k}}$ and thus $\mA_k=\mE_k\#H$ with $\cE_k\simeq \mA_k^{\co H}$. Moreover, $\cE_k$ is an $H$-module algebra \cite[Proposition 5.8 (d)]{AAGMV} and the map $p:\mA_k\to \mE_k$, given by $p(x)=x\_{0}\gamma_k^{-1}\iota\pi(x\_{1})$, defines an $H$-module projection. 
Notice that $\mE_0=T(V)$.

\begin{fed}\label{def:varphi A E}
For each $\bs\lambda=(\lambda_i)_{i=1}^n\in\ku^n$, we define the set
\begin{align*}
\cG_k(\bs\lambda)&=\{ \lambda_i +\gamma_k(\Ss(u_i)) : i\in\I_n \}\subset \mE_k.
\end{align*}
Let $\cI(\bs\lambda)=\lg \cG_k(\bs\lambda)\rg\subset \mE_k$ and set $\mE(\bs\lambda)=\mE_k/\cI(\bs\lambda)$. 
We consider the following conditions:
\begin{align}
\label{eqn:cond1}\mE(\bs\lambda)&\neq 0;\\
\label{eqn:cond2}\ku\cG_k(\bs\lambda)&\subset\mE_k \text{ is an }H\text{-submodule}; \\
\label{eqn:cond3}\varphi_{\bs\lambda}& \text{ extends to an algebra map }Y_k\rightarrow\mA_k.
\end{align}
The $k$th set of deforming parameters is
\begin{align}\label{eqn:Lambda}
\bs\Lambda_k=\bigl\{\bs\lambda\in\ku^n: \eqref{eqn:cond1}, \eqref{eqn:cond2} \text{ and } \eqref{eqn:cond3}\text{ hold}\bigr\}.
\end{align}
Thus $\mE(\bs\lambda)$ is an $H$-module algebra, if $\bs\lambda\in\bs\Lambda_k$. We set  $\mA(\bs\lambda)=\mE(\bs\lambda)\# H$.
\end{fed}

\begin{rem}\label{rem:free}
Assume that the elements of $\Gc_0$ are homogeneous of the same degree. Then $\varphi_{\bs\lambda}:Y_0\longrightarrow\mA_0$ is an algebra map for all $\bs\lambda\in\ku^n$.

Indeed, the subalgebra of $T(V)$ generated by $\Gc_0$ is free by \cite[Lemma 28]{AV2}. Since the antipode is an anti-homomorphism of algebras, $Y_0$ is also free.
\end{rem}

The algebras which we shall collect in $\Lambda_{k+1}$ are precisely the algebras $\mA(\bs\lambda)$, $\bs\lambda\in\bs\Lambda$. Indeed, we have the following.

\begin{lem}
$\mA(\bs\lambda)\in\Cleft(\mH_{k+1})$ for all $\bs\lambda\in\bs\Lambda_k$. If $\mA_{k+1}\in \Lambda_{k+1}$, then there is $\bs\lambda\in\bs\Lambda_k$ such that $\mA_{k+1}=\mA(\bs\lambda)$.
\end{lem}
\pf
Set $s_i\coloneqq\Ss(u_i)=-\sum_{j=1}^n\Ss(E_{ij})u_j$, $i\in\I_n$. Using the fact that $\gamma_{k\mid H}^{-1}=\Ss_H$ and $\ad_{\mid H}$ coincides with the action of $H$ on $\B_k$ inside $\ydh$, it follows that
\begin{align}\label{eqn:yi}
p\varphi_{\bs\lambda}(s_i)=\lambda_i-\sum_{l=1}^n\gamma_k(\Ss(E_{il})\cdot u_l)=\lambda_i+\gamma_k(\Ss(u_i))\in\Gc_k(\bs\lambda).
\end{align} 
Then $p$ induces a linear isomorphism $\ku\varphi_{\bs\lambda}(\Gc_k)\longrightarrow\ku\Gc_k(\bs\lambda)$. In particular, $\ku\varphi_{\bs\lambda}(\Gc_k)$ is an $H$-submodule of $\mA_k$. Hence $\langle\varphi_{\bs\lambda}(Y_k^+)\rangle=\cI(\bs\lambda)\#H$ by \cite[Remark 5.6 (d)]{AAGMV} and therefore $\mA(\bs\lambda)=\mA_k/\langle\varphi_{\bs\lambda}(Y_k^+)\rangle\simeq\mE(\bs\lambda)\#H\neq0$. By \cite[Theorem 3.3]{AAGMV} $\mA(\bs\lambda)\in\Cleft(\mH_{k+1})$ because $\varphi_{\bs\lambda}\in\Alg^{\cH_k}(Y_k,\cA_{k})$.

Now, let $\mA_{k+1}\in \Lambda_{k+1}$ and recall that $\mA_{k+1}=\mA_k/\langle\varphi(Y_k^+)\rangle$ for some $\varphi\in\Alg^{\cH_k}(Y_k,\cA_{k-1})$, see \cite[\S 5.2 (b)]{AAGMV}. Hence, \cite[Lemma 5.9 (a)]{AAGMV} states that for such $\varphi$ there is $\bs\lambda\in\ku^n$ such that $\varphi_{|\Ss(\Gc_k)}=\varphi_{\bs\lambda}$
and thus $\mA_{k+1}=\mA_k/\langle\varphi_{\bs\lambda}(s_i): i\in\I_n\rangle\simeq \mA(\bs\lambda)$.
\epf

Under a feasible assumption, we can remove the antipode from the relations defining $\mE(\bs\lambda)$.

\begin{lem}\label{le:gen of rel of Ek+1 mejorado}
Let $\bs\lambda\in\bs\Lambda_k$. If $\gamma_k$ is a morphism of $H$-modules, then
\begin{align*}
\mE(\bs\lambda)\simeq \mE_k(\bs\lambda)/\langle \lambda_i-\gamma_k(u_i):i\in\I_n\rangle.
\end{align*}
\end{lem}

\begin{rem}
The section $\gamma_k$ is $H$-linear when
\begin{enumerate}
\item $H$ semisimple, equivalently $\dim H<\infty$, or
\item $k=0$.
\end{enumerate}
Indeed, (1) is \cite[Proposition 5.8 (c)]{AAGMV} and (2) holds since $\gamma_0=\id$.
\end{rem}

\begin{proof}
It follows from the $H$-linearity of $\gamma_k$ that $\gamma_k(u_i)\in\mE_k$, $i\in\I_n$, arguing as in \cite[Lemma 4.1]{AGI}. 
It is enough to see that $\ku\Gc_k(\bs\lambda)=\ku\{\lambda_i-\gamma_k(u_i)\}_{i\in\I_n}$ as $H$-submodules by \cite[Remark 5.6 (b)]{AAGMV}.

Let $z_i\coloneqq p\varphi_{\bs\lambda}(s_i)$ for $i\in\I_n$, recall \eqref{eqn:yi}. If $s\in\I_n$, then
\begin{align*}
\sum_{i}E_{s,i}\cdot z_i&=\sum_{i}E_{s,i}\cdot \lambda_i-\sum_{i,l}\bigl(E_{s,i}\Ss(E_{il})\bigr)\cdot\gamma_k(u_l)\\
&=\sum_{i}\varepsilon(E_{s,i}) \lambda_i-\sum_{l}\varepsilon(E_{s,l})\gamma_k(u_l)\\
&=\lambda_s-\gamma_k(u_s).
\end{align*}
Reciprocally, $z_i=\sum_{l}\Ss(E_{i,l})\cdot(\lambda_l-\gamma_k(u_l))$ for all $i\in\I_n$.
\end{proof}

\subsection{On step 2}\label{sec:on step 2}
By construction, $L(\mA_{k+1},\mH_{k+1})$ is a cocycle deformation of $\mH_{k+1}$. Moreover, if $\F=(F_n)_{n\geq 0}$ is the filtration on $L(\mA_{k+1},\mH_{k+1})$ induced by the graduation of $\mH_{k+1}$, then the associated graded Hopf algebra $\gr_{\F}L(\mA_{k+1},\mH_{k+1})$  is isomorphic to $\mH_{k+1}$ \cite[Proposition 4.14 (c)]{AAGMV}; hence $L(\mA_{N+1},\mH_{N+1})$ is a lifting of $V$. Thus, the set of deforming parameters
\[
\bs\Lambda=\bs\Lambda_0\times \dots\times \bs\Lambda_N
\]
parametrizes a family of liftings $\mH(\bs\lambda)$ of $V$, $\bs\lambda=(\bs\lambda_0,\dots,\bs\lambda_N)\in\bs\Lambda$.

The Hopf algebras $\mL_1(\bs\lambda)=L(\mA_1(\bs\lambda),\mH_1)$, $\bs\lambda\in\bs\Lambda_0$, obtained after the first recursion are easy to describe. We set
\begin{align*}
\mH_1(\bs\lambda)=T(V)\#H/\langle u_i- \lambda_i+\sum_{\mathclap{j\in\I_n}}\lambda_jE_{ij}: i\in\I_n \rg.
\end{align*}

The next lemma is a particular case of \cite[Proposition 5.10]{AAGMV}. 

\begin{lem}\label{le:L1}
Let $\bs\lambda\in\bs\Lambda$. Then
\begin{align*}
L(\mA_1(\bs\lambda),\mH_1)\simeq\mH_1(\bs\lambda).
\end{align*}
\end{lem}

\begin{proof}
Set $\mL_1=\mL_1(\bs\lambda)$ and let $s_i=\Ss(u_i)$, $i\in\I_n$. 
By \cite[Proposition 5.10]{AAGMV}, we have that
\begin{align*}
\cL_1\simeq\mL_{0}/\langle \widetilde s_i -
c_i+\sum_{j=1}^nc_j \Ss(E_{ij})\rangle_{1\leq i \leq n}
\end{align*}
where $\widetilde s_i\in\mL_{k}$ $1\leq i \leq n$, is such that
\begin{align*}
 \widetilde s_i\ot
1_{\cA_0}=\sum_{t}\gamma_0(s_t)\_{-1}\ot \gamma_0(s_t)\_{0}\gamma_0^{-1}\Ss(E_{it}
)+1\ot\gamma_0^{-1}(s_i).
\end{align*}
This formula is simplified by the fact that $\gamma_0=\id$, $\gamma^{-1}_0=\Ss$ and the coaction is the comultiplication in $\cL_0$.  Notice that, in particular,
\[ \Delta(s_i)=\sum_{j=1}^ns_j\ot\Ss(E_{ij})+1\ot s_i, \qquad i\in\I_n.\] 
Hence, 
\begin{align*}
\sum_{t}(s_t)\_{-1}&\ot (s_t)\_{0}\Ss^2(E_{it})+1\ot\Ss(s_i)=\\
=&\sum_{j,t}s_j\ot\Ss(E_{tj})\Ss^2(E_{it})+\sum_{t}1\ot s_t\Ss^2(E_{it})+1\ot\Ss(s_i)\\
=&\sum_{j}s_j\ot\varepsilon\Ss(E_{ij})+\sum_{t,j}-1\ot \Ss(E_{tj})u_j\Ss^2(E_{it})+1\ot\Ss(s_i)\\
=&s_i\ot1+1\ot\left(\Ss^2(u_i)-\sum_{j}\ad\Ss(E_{ij})(u_j)\right).
\end{align*}
Then, the second summand has to be zero and $\widetilde s_i=s_i=\Ss(u_i)$. 
\end{proof}

\subsection{On step 3}\label{sec:on step 3}

We give a sufficient condition to find all the quadratic liftings via cocycle deformation.

\begin{lem}\label{thm:liftings-step-3} 
Assume that $\ku\Gc=\mJ^2(V)$ and let $L$ be a lifting of $V$. 
If $\homh(\mJ^2(V),V)=0$, then there exists $\bs\lambda\in\bs\Lambda$ such that $L\simeq\mH(\bs\lambda)$.
\end{lem}

\begin{rem}
In the setting of Lemma \ref{thm:liftings-step-3}, $\bs\Lambda=\bs\Lambda_0$ and $\mH(\bs\lambda)=\mH_1(\bs\lambda)$. 
\end{rem}

\pf
Let $\phi:T(V)\# H\twoheadrightarrow L$ be a lifting map. 
Now, $M=\mJ^2(V)$ is compatible with $\phi$ \cite[Definition 4.7]{AAGMV} since $\mJ^2(V)\subset P(T(V))$. Following \cite[Lemma 4.8]{AAGMV}, $\pi_1\circ\phi_{|M}=0$ by our hypothesis. Then there are scalars $\bs\lambda=\{\lambda_i\}_{i=1}^n$ such that $L\simeq \mL_1(\bs\lambda)$ by \cite[Lemma 4.8 (c) and Theorem 4.11]{AAGMV}. Hence the lemma follows from Lemma \ref{le:L1}.
\epf

\begin{rem}
The space $\homh(\mJ(V),V)$, for $V$ a braided vector space of diagonal type, has been studied in \cite{AKM} in connection with braided deformations of $\B(V)$.
\end{rem}

\begin{rem}
In the setting of Lemma \ref{thm:liftings-step-3}, we see every lifting is determined by a family $\bs\lambda\in\bs\Lambda$. However, two different families $\bs\lambda$ and $\bs\lambda'$ may define the same, or isomorphic, liftings. 

For each $V,H$, we study the symmetries of the set $\bs\Lambda$ to describe the complete list of non-isomorphic liftings, see e.g.~\cite[Lemma 4.8 (d)]{AAGMV}. 
\end{rem}

\section{Realizations}\label{sec:realizations}
We fix a rack and 2-cocycle $(X,q)$. We investigate some of the necessary conditions on a Hopf algebra $H$ so that we can realize the braided vector spaces  $V(X,q)$ and $W(q,X)$ in ~\S 
\ref{sec:racks-bvs}.

\subsection{Realizations of $V(X,q)$}

We shall study the class of Hopf algebras $H$ with a realization $V(X,q)\in\ydh$. 

\begin{exa}\label{exa:enveloping}
If $H=\k G$, $G$ a finite group, a principal realization of $V(X,q)$ over $H$ \cite[Definition 3.2]{AG2} is the data $(\cdot, g, (\chi_i)_{i\in X})$ given by:
\begin{itemize}[leftmargin=*]
\item $\cdot:G\times X\to X$ is an action of $G$ on $X$; 
\item $g:X\to G$ is a function such that $g(h\cdot i) = hg(i)h^{-1}$
and $g(i)\cdot j=i\rhd j$; 
\item $\chi_i:G\to\k^\times$ satisfies $\chi_i(g(j))=q_{ji}$, $i,j\in X$, and the family $(\chi_i)_{i\in X}$ is a 1-cocycle, i.e.~$\chi_i(ht)=\chi_i(t)\chi_{t\cdot i}(h)$, for all $i\in X$, $h,t\in
G$.
\end{itemize}
This data defines a Yetter-Drinfeld structure on $V(X,q)$ by 
\[
v_x\mapsto g_x\ot v_x, \qquad g\cdot v_x=\chi_x(g)v_{g\cdot x}; \quad x\in X, g\in G.
\]
The realization is said to be \emph{faithful} if $g$ is injective, which is always the case when $X$ is faithful \cite[Lemma 3.3]{AG1}. 

Notice that $V(X,q)$ clearly admits a natural principal realization over the enveloping groups $G_X$ and $F_X$, see \S \ref{sec:racks}.

\end{exa}

Let $H$ be a Hopf algebra. Assume that $V=V(X,q)$ is an $H$-module and an $H$-comodule: that is there are matrix coefficients $\{\mu_{xy}\}_{x,y\in X}\subset H^*$ and comatrix elements $\{e_{xy}\}_{x,y\in X}\subset H$ such that the action and coaction are determined respectively by
\begin{align}\label{eqn:action-coaction}
h\cdot v_x&=\sum_{y\in X} \mu_{xy}(h) v_y, & \lambda(v_x)&=\sum_{y\in X}e_{xy}\ot v_y.
\end{align}

\begin{lem}\label{lem:realization}
Equations \eqref{eqn:action-coaction} define a realization of $V$ over $H$ if and only if 
\begin{align}\label{eqn:yd-cond}
\sum_{y\in X} \mu_{xy}(h_{(1)}) e_{yz}h_{(2)}=\sum_{y\in X} \mu_{yz}(h_{(2)}) h_{(1)}  e_{xy}, \qquad \forall\, x,z\in X, h\in H.
\end{align}
\end{lem}
\pf
This is a translation of the Yetter-Drinfeld compatibility.
\epf
We shall fix
\begin{align}\label{eqn:subalgebra}
K\coloneqq \text{ Hopf subalgebra of }H\text{ generated by }\{e_{xy}\}_{x,y\in X}.  
\end{align}

\begin{rem}\label{rem:ppal}
In the context of Lemma \ref{lem:realization}, assume that the realization $V\in \ydh$ is principal. That is, $e_{xy}=\delta_{x,y}g_{x}$, with $g_x\in G\coloneqq G(H)$. Then \eqref{eqn:yd-cond} becomes: 
\[
\mu_{xy}(h_{(1)}) g_y h_{(2)}= \mu_{xy}(h_{(2)}) h_{(1)}  g_x, \qquad \forall\, x,z\in X, h\in H.
\]
In particular, the realization restricts to $V\in \ydg$. Hence, if the rack is faithful, then this restriction is given by a principal 
realization. 
\end{rem}

We fix a realization of $V$ over $H$. The action takes a simpler form when restricted to the subalgebra $K\subseteq H$ as in \eqref{eqn:subalgebra}.

\begin{lem}\label{le:action of comatrix on V}
For $x,y,z,t\in X$, 
\begin{align*}
\mu_{xy}(e_{zt})=\delta_{y,z\rhd x}\delta_{z,t}q_{z,x}.
\end{align*}
The action of $K$ on $V$ is given by
\begin{align*}
\Ss^n(e_{xy})\cdot v_z=\delta_{x,y}\,f^{(-1)^n}_x\cdot v_z, \quad n\geq 0.
\end{align*}
\end{lem}
\begin{proof}
By definition of realization we have that
\begin{align*}
q_{x,y}v_{x\rhd z}\ot v_x=\sum_{y\in X}e_{xy}\cdot v_z\ot v_y, \quad \forall\, x,z\in X.
\end{align*}
From here we deduce the formula for $\mu_{xy}(e_{zt})$. Now, for $n\geq1$ and $x\neq y$,
\begin{align*}
v_z&=\varepsilon(\Ss^{n}(e_{xx}))v_z=\sum_{t\in X}\Ss^{n+1}(e_{tx})\Ss^n(e_{xt})\cdot v_z=\Ss^{n+1}(e_{xx})\Ss^n(e_{xx})\cdot v_z\\
&=\Ss^{n+1}(e_{xx})f^{(-1)^n}_{x}\cdot v_z,\\
\noalign{\smallskip}
0&=\varepsilon(\Ss^n(e_{xy}))v_{\phi^{(-1)^n}_x(z)}=\sum_{t\in X}\Ss^{n+1}(e_{ty})\Ss^n(e_{xt})\cdot v_{\phi^{(-1)^n}_x(z)}\\
&=\Ss^{n+1}(e_{xy})\Ss^n(e_{xx})\cdot v_{\phi^{(-1)^n}_x(z)}
\end{align*}
and the lemma follows.
\end{proof}

\begin{cor}\label{cor:relationK}
The following relations hold in $K$, for all $x,y,s,t\in X$,
\begin{align}\label{eqn:yd-cond-K}
q_{t,y}e_{st}e_{xy}=q_{s,x}e_{s\rhd x,t\rhd y}e_{st}.
\end{align}
\end{cor}
 \pf
Follows by plugging $h=e_{st}$ into \eqref{eqn:yd-cond}.
\epf

Recall that $G_X=\lg \{f_x\}_{x\in X}| f_xf_y=f_{x\rhd y}f_x\rg$ denotes the enveloping group of the rack $X$, \S \ref{sec:racks}.

\begin{pro}\label{pro:oK}
The quotient 
\begin{align*}
\oK=K/\langle\Ss^n(e_{xy})\mid x\neq y,\, n\in\N\rangle
\end{align*}
is a non-zero group algebra quotient of $\ku G_X$, via $f_x\mapsto\ke_{xx}$, $x\in X$. We set $\oG=\lg\ke_{xx}:x\in X\rg$, so $\oK=\k\oG$ and $V\in{}^{\oG}_{\oG}\mathcal{YD}$. 
If $X$ is faithful, then $\{\ke_{xx}\}_{x\in X}$ is a linearly independent set. 
\end{pro}

\begin{proof}First, as the elements $\{e_{xy}\}_{x\neq y}$ span a coideal contained in $K^+$, $\oK$ 
is a Hopf algebra. Now, $\ke_{xx}$ is a group-like element with inverse $\Ss(\ke_{xx})$, $\oK$ is non-zero by Lemma \ref{le:action of comatrix on V} and $V\in{}^{\oK}_{\oK}\mathcal{YD}$. 
In $\oK$, it holds $\ke_{xx}\ke_{zz}=\ke_{(x\rhd z)(x\rhd z)}\ke_{xx}$ for all $x,y\in X$, by \eqref{eqn:yd-cond-K}; and the assignment $f_x\mapsto\ke_{xx}$, $x\in X$, extends to a Hopf algebra map.
\end{proof}

\begin{rem}\label{rem:ee}
Let $C=\{(i_2,i_1),\dots\}\in\mR$ and let $e_C=e_{i_2i_2}e_{i_1i_1}$. Then
\[
e_{i_{h+1}i_{h+1}}e_{i_hi_h}=e_C, \quad 1\leq h<|C|.
\]
Indeed, $e_{i_{h+1}i_{h+1}}e_{i_hi_h}=e_{i_{h}\rhd i_{h-1}i_{h}\rhd i_{h-1}}e_{i_hi_h}=e_{i_hi_h}e_{i_{h-1}i_{h-1}}$ by \eqref{eqn:yd-cond-K}.
\end{rem}

We have seen in Example \ref{exa:enveloping} that any braided vector space $V=V(X,q)$ of rack type can be realized over some group algebra $\ku G$. Now, even though the categories $\ydg$ and $\ydgdual$ are braided equivalent \cite{AG1}, we may not be able to realize this $V$ over $\ku^G$. The following lemma exemplifies this situation in a concrete case.

\begin{lem}\label{lem:non-real}
$V(\cO_2^3,-1)$ cannot be realized over $H=\ku^{\Sn_3}$. 
\end{lem}
\pf
Indeed, the Hopf subalgebras of $K\subset H$ are $K=H$ and $K=\ku\langle\sgn\rangle$, and neither of these have projections onto quotients of $\ku G_{\cO_2^3}$ with $g_{(12)}$, $g_{(23)}$ and $g_{(13)}$ linearly independent.
\epf

\subsubsection{On the quadratic relations}

Notice that, as $V\in \ydh$, the space $\k\{b_C\}_{C\in\mR'}\subset T(V)$ is a Yetter-Drinfeld submodule, recall \eqref{eqn:qrels}.  In particular, there are matrix coefficients 
$\alpha_{C,D}\in H^*$ and comatrix elements $E_{C,D}\in H$ cf.~\eqref{eqn:comatrix-relations}, $C,D\in \mR'$, so that
\begin{align}\label{eqn:action-bc}
h\cdot b_C=\sum_{D\in\mR'}\alpha_{C,D}(h)b_D, \qquad b_C\mapsto \sum_{D\in\mR'}E_{C,D}\ot b_D.
\end{align}
In the next lemma, we express these structural data in terms of the matrix coefficients $\{\mu_{xy}\}_{x,y\in X}$ and comatrix elements $\{e_{xy}\}_{x,y\in X}$  in \eqref{eqn:action-coaction}.

Recall that $\ast$ denotes the convolution product in $H^*$ and the notation $x\rhd C[h]$ from \S \ref{subsec:quadratic}.

\begin{lem}\label{lem:structure-b}
Let $C=\{(i_2,i_1),\dots\}$ and $D=\{(j_2,j_1),\dots \}\in\mR'$. Fix $l=1,\dots,|D|$, then   
\begin{align*}
\alpha_{C,D}&=\sum_{h=1}^{|C|}\frac{\eta_h(C)}{\eta_l(D)}\mu_{i_{h+1}j_{l+1}}\ast \mu_{i_hj_l},& E_{C,D}&=\sum_{h=1}^{|C|}\frac{\eta_h(C)}{\eta_l(D)} e_{i_{h+1}j_{l+1}}e_{i_hj_l}.
\end{align*}
In particular,
\begin{align}\label{eqn:alpha-exy}
\alpha_{C,D}(e_{x,y})&=\delta_{x,y}\sum_{h=1}^{|C|}\frac{\eta_h(C)}{\eta_l(D)}\delta_{D[l],x\rhd C[h]}
q_{x,i_{h+1}}q_{x,i_{h}}
\end{align}
If $(s,t)\notin D$, then 
\begin{align*}
\sum_{h=1}^{|C|}\eta_h(C)\mu_{i_{h+1}s}\mu_{i_ht}=0=\sum_{h=1}^{|C|}\eta_h(C) e_{i_{h+1}s}e_{i_ht}.
\end{align*}
\end{lem}

\begin{proof}
The structural data of the Yetter-Drinfeld module $T^2(X,q)$ in the basis $\{v_xv_y\}_{x,y\in X}$ are given by 
\begin{align*}
\beta_{xy,st}=\mu_{xs}\ast\mu_{yt}\quad\mbox{and}\quad d_{xy,st}=e_{xs}e_{yt},
\end{align*}
 Thus, we can compute the action and coaction of the $b_C$'s in two different ways: using either the matrix coefficients $\beta_{xy,st}$ and the comatrix elements $d_{xy,st}$ or the $\alpha_{C,D}$'s and $E_{C,D}$'s. The lemma follows by comparing both computations.
Identity \eqref{eqn:alpha-exy} follows by Lemma \ref{le:action of comatrix on V}.
\end{proof}

\begin{rem}
We stress that $\mu_{C,D}$ and $E_{C,D}$ can be expressed in several ways, as many as the cardinal of $D$. If we choose $l=1$, then recall that $\eta_1(D)=1$. In particular, \eqref{eqn:alpha-exy} becomes
\begin{align}\label{eqn:alpha-HK}
\alpha_{C,D}(e_{xy})=
\begin{cases}
\delta_{x,y}\eta_h(C)q_{x,i_{h+1}}q_{x,i_h} & \text{if }(j_2,j_1)=x\rhd(j_{h+1},j_h),\\
0& \text{otherwise}.
\end{cases}
\end{align}
\end{rem}

\begin{rem}\label{rem:gC}
In the setting of Remark \ref{rem:ppal}, that is when $V\in\ydh$ is principal, Lemma \ref{lem:structure-b} univocally associates an element $g_C=g_{i_2}g_{i_1}\in G(H)$ to each $C\in\mR'$ in such a way that $\lambda(b_C)=g_C\ot b_C$; cf.~\cite[\S 3.3]{AG2}.  
\end{rem}

\subsection{Realizations of $W(q,X)$}\label{subsec:WqX}
Let $G$ be a finite group. Recall that there is a  braided equivalence $(F,\eta):\ydg\to \ydgdual$ \cite[Proposition 2.2.1]{AG1} with $F(V)=V$ as vector spaces and action and coaction given 
by:
\begin{align}\label{prop:equiv de cat gdual}
\begin{split}
f\cdot v&=\langle f,\cS(v\_{-1})\rangle 
v\_{0},\quad\lambda(v)=\sum_{t\in G}\delta_t\ot t^{-1}\cdot v\quad\mbox{and}\\
\noalign{\smallbreak}
\eta:\, &F(V)\ot F(U)\longmapsto F(V\ot U),\quad v\ot u\mapsto u_{(-1)}v\ot u_{(0)}
\end{split}
\end{align}
for every $V,U \in \ydg$, $f \in \ku^G$, $v \in V$, $u\in U$.

\begin{lem}\label{lem:braided}
Assume that $V(X,q)$ has a principal realization over $G$ and let $W$ be the image of $V(X,q)$ by $F$. Then  $W=W(q,X)$ as braided vector spaces.
\end{lem}
\pf According to \eqref{prop:equiv de cat gdual}, $W=V$ as vector spaces,
\begin{align*} 
\delta_t\cdot v_x=\delta_{t,g_x^{-1}}v_x\quad\mbox{and}\quad \lambda(v_x)=\sum_{t\in G}\chi_x(t^{-1})\delta_t\ot v_{t^{-1}\cdot  x}, \, t\in G,\, x\in X.
\end{align*}
Hence, the braiding is given by  
\begin{align*}
c(v_x\ot v_y)
&=\sum_{t\in G}\chi_x(t^{-1})\delta_t \cdot v_y\ot v_{t^{-1}\cdot  x}
=\chi_x(g_y) v_y\ot v_{g_y\cdot  x}\\
&=q_{y,x} v_y\ot v_{y\rhd x}.
\end{align*}
Thus, we can identify $W$ with $W(q,X)$ via $v_x\mapsto w_x$.
\epf

In view of the Lemma \ref{lem:braided}, we also call a {\it principal realization of $W(q,X)$ over $\ku^G$} to the image by the functor $(F,\eta)$ of the Yetter-Drinfeld module $V(X,q)\in\ydg$ defined by a datum $(\cdot, g, (\chi_i)_{i\in X})$.

\subsection{Quadratic relations of $\B(q,X)$}\label{sec:quadratic-q-X}

Recall that the quadratic relations of $\B(X,q)$ are generated by certain elements $b_C$, $C\in\mR'= \mR'(X,q)$, see \S \ref{subsec:quadratic}. 
For every $C\in \mR'(X,q)$, we define the element $\tilde b_C\in T(q,X)$ by
\begin{align}\label{eqn:qrels para  W q X}
\tilde b_C:&= \sum_{h=1}^{|C|}\eta_h(C)\, w_{i_h}w_{i_{h+1}}.
\end{align}

\begin{pro}\label{pro:quadratic-dual}
The quadratic relations of $\B(q,X)$ are $$\mJ^2(q,X)=\ku\{ \tilde b_C:C\in \mR' \}.$$
\end{pro}

\begin{proof}
We can consider $V(X,q)$ as a Yetter-Drinfeld module over a finite group $G$ with a principal realization \cite{AG3}, see Example \ref{exa:enveloping}. Then, by \cite[Lemma 3.2]{GV}, $\mJ^2(q,X)=\eta^{-1}(\mJ^2(X,q))$. Using that
\begin{align*}
g_j^{-1}\cdot v_i=q^{-1}_{j,\phi^{-1}_j(i)}v_{\phi^{-1}_j(i)}\quad\mbox{and}\quad\prod_{h=1}^{|C|} q_{i_{h+1},i_h} =(-1)^{|C|},
\end{align*}
we obtain $\eta^{-1}(b_C)=\tilde b_C$ and the lemma follows.
\end{proof}

The next, well-known fact, is a direct consequence of Proposition \ref{pro:quadratic-dual}.

\begin{cor}\label{cor:nichols}
Let $n=3,4,5$.
\begin{enumerate}\renewcommand{\theenumi}{\alph{enumi}}
\renewcommand{\labelenumi}{(\theenumi)}
\item 
The ideal $\cJ(-1,X_n)$ is generated by $\xij{ij}^2$ and \eqref{eq:Vn},
for all $(ij)\in\cO_2^n$. 

\smallskip
\item The ideal $\cJ(\chi,X_n)$ is generated by $\xij{ij}^2$ and \eqref{eq:Wn},
for all $(ij)\in\cO_2^n$. 

\smallskip
\item 
The ideal $\cJ(Y,-1)$ is generated by $x_\sigma x_{\sigma^{-1}}+x_{\sigma^{-1}}x_\sigma$ and \eqref{eq:U}
for all $\sigma\in\cO_4^4$.
\qed\end{enumerate}
\end{cor}

Let $H$ be a Hopf algebra and fix a realization of $W(q,X)$ over $H$. Let $\{\mu_{xy}\}_{x,y\in X}\in H^*$ and $\{e_{xy}\}_{x,y\in X}\in H$ be the  associated matrix coefficients and comatrix elements of this realization, cf.~\eqref{eqn:action-coaction}. Then they must satisfy \eqref{eqn:yd-cond}. As in \eqref{eqn:subalgebra} we still denote by $K$ the Hopf subalgebra of $H$ generated by $\{e_{xy}\}_{x,y\in X}$.

\begin{lem}
For all $x,y,z,t\in X$, it holds that $\mu_{zt}(e_{xy})=\delta_{z,t}\,\delta_{z\rhd x,y}\,q_{zx}$ and hence
$q_{y,t}\,e_{st}\,e_{xy}=q_{x,t}\,e_{xy}\,e_{x\rhd s, y\rhd t}$.

Moreover, the action of $K$ on $W(q,X)$ satisfies
\begin{align*}
\Ss^n(e_{xy})\cdot w_z=\begin{cases}
                       \delta_{z\rhd x,y}\,q_{zx}\, w_z,&\mbox{if $n$ is even,}\\
                       \delta_{x,z\rhd y}\,q^{-1}_{zy}\, w_z&\mbox{if $n$ is odd}.
                       \end{cases}
\end{align*}
\end{lem}

\pf
It follows as Lemma \ref{le:action of comatrix on V} and \eqref{eqn:yd-cond-K}.
\epf

We see that every $\k\{w_z\}\subset W(q,X)$ is $K$-invariant. This defines, for each $z\in X$, an algebra map 
\[
\vartheta_z\coloneqq\mu_{zz}\in\Alg(K,\ku); \qquad e_{xy}\mapsto \delta_{z\rhd x,y}\,q_{zx}, \ x,y\in X.
\]

\begin{pro}\label{prop:oK dual}
\begin{enumerate}
\item For every $z,t\in X$, $\vartheta_z\vartheta_t=\vartheta_t\vartheta_{t\rhd z}$.
\item $G_X^{op}\rightarrow\Alg(K,\ku)$, $f_x\mapsto\vartheta_x, x\in X$, is a group homomorphism. 
\item If $X$ is faithful, then $\{\vartheta_z\}_{z\in X}$ is linearly independent.
\end{enumerate}

\end{pro}

\pf
(1) By an straightforward computation, we obtain that
\begin{align*}
\vartheta_z\vartheta_t(e_{xy})&=\begin{cases}
                               q_{zx}\,q_{t,z\rhd x},&\mbox{if $t\rhd(z\rhd x)=y$,}\\
                               0,&\mbox{otherwise;}
                               \end{cases}
                               \\
\vartheta_t\vartheta_{t\rhd z}(e_{xy})&=\begin{cases}
                               q_{tx}\,q_{t\rhd z,t\rhd x},&\mbox{if $(t\rhd z)\rhd(t \rhd x)=y$,}\\
                               0,&\mbox{otherwise.}
                               \end{cases}
\end{align*}
Since $\rhd$ is self distributive, we only have to check $q_{zx}\,q_{t,z\rhd x}=q_{tx}\,q_{t\rhd z,t\rhd x}$ and this follows from the definition of $1$-cocycle. The proof of 
$\vartheta_z\vartheta_t(\Ss^n(e_{xy}))=\vartheta_t\vartheta_{t\rhd z}(\Ss^n(e_{xy}))$, $n\in\N$, is similar. Thus (2) also follows.

If $X$ is faithful, it follows that $\vartheta_z\neq\vartheta_t$ for $z\neq t$ and thus they are linearly independent in $K^*$ since they are group-like elements in the Hopf algebra $K^\circ\subseteq K^*$ (the Sweedler dual of $K$).
\epf

\begin{rem}
Proposition \ref{prop:oK dual} gives a necessary condition to realize $W(q,X)$ over a Hopf algebra. We deduce that $W(\sgn,\cO_2^3)$ cannot be realized over $\ku{\Sn_3}$; compare it with 
Lemma \ref{lem:non-real}. 
\end{rem}

Let $\{\tilde\alpha_{C,D}\}_{C,D\in\mR'}$ and $\{\tilde E_{C,D}\}_{C,D\in\mR'}$ be the structural data of $\mJ^2(q,X)$ as Yetter-Drinfeld module over $H$. 
That is, let
\begin{align}\label{eqn:YDH structure of quadratic rel qX}
h\cdot\tilde b_C=\sum_{D\in\mR'}\tilde\alpha_{C,D}(h)\tilde b_D, \qquad \tilde b_C\mapsto \sum_{D\in\mR'}\tilde E_{C,D}\ot \tilde b_D
\end{align}
define the $H$-action and $H$-coaction, respectively. Next lemma follows as Lemma \ref{lem:structure-b} and provides formulas to compute these data in terms of the realization. 

\begin{lem}
Let $C=\{(i_2,i_1),\dots\}$ and $D=\{(j_2,j_1),\dots \}\in\mR'$. Fix $l=1,\dots,|D|$, then 
\begin{align*}
\tilde\alpha_{C,D}&=\sum_{h=1}^{|C|}\frac{\eta_h(C)}{\eta_l(D)}\mu_{i_hj_l}\ast \mu_{i_{h+1}j_{l+1}}, & 
\tilde E_{C,D}&=\sum_{h=1}^{|C|}\frac{\eta_h(C)}{\eta_l(D)} e_{i_hj_l}e_{i_{h+1}j_{l+1}}.
\end{align*}
If $(s,t)\notin D$, then $
\sum\limits_{h=1}^{|C|}\eta_h(C)\mu_{i_hs}\ast\mu_{i_{h+1}t}=0=\sum\limits_{h=1}^{|C|}\eta_h(C) e_{i_hs}e_{i_{h+1}t}$. \qed
\end{lem}
\section{Liftings of $V(X,q)$}\label{sec:pointed}

In this section we shall compute the liftings of a realization $V=V(X,q)\in\ydh$ for $(X,q)$ a rack and a cocycle from the list in \S \ref{sec:list}.

\subsection{A general setting for quadratic deformations}
The Nichols algebras associated to the pairs $(X,q)$ listed in \S \ref{sec:list} are quadratic, see Theorem \ref{thm:nichols}.
This motivates the following general definition. 

Let $(X,q)$ be a finite rack and a 2-cocycle, set $\B(X,q)$ as in \S \ref{sec:racks-bvs}. Recall the description of the generators $\{b_C\}_{C\in \mR'}$ of the space of quadratic relations 
in $\B(X,q)$ from \eqref{eqn:qrels}.  Let $H$ be a cosemisimple Hopf algebra supporting a realization $V=V(X,q)\in\ydh$. 

Recall the strategy presented in \S \ref{sec:strategy}; we shall study the set of deforming parameters $\bs\Lambda_0\subseteq\ku^{\mR'}$  for $\Gc_0=\{b_C\}_{C\in \mR'}$. We write $\bs\Lambda=\bs\Lambda_0$ for short. We follow Definition \ref{def:varphi A E} and Lemma \ref{le:gen of rel of Ek+1 mejorado} to fix
\begin{align*}
\Gc(\bs\lambda)&=\{ b_C - \lambda_C : C\in \mR'\}\subset T(V), & \mE(\bs\lambda)&=T(V)/\lg \Gc(\bs\lambda) \rg.
\end{align*}
Recall the notation from \eqref{eqn:action-bc} and Lemma \ref{lem:structure-b}.

\begin{lem}
$\mE(\bs\lambda)$ inherits an $H$-module structure from $T(V)$ when
\begin{align}\label{eqn:lambda1}
\eps(h)\lambda_C=\sum_{D\in\mR'}\alpha_{C,D}(h)\lambda_D, \qquad C\in\mR', h\in H. \qed
\end{align}
\end{lem}
That is, \eqref{eqn:cond2} is equivalent to \eqref{eqn:lambda1}. Hence the set of deformation parameters \eqref{eqn:Lambda}, see also Remark \ref{rem:free}, is given by
\begin{align*}
\bs\Lambda=\{\bs\lambda=(\lambda_C)_{C\in \mR'} | \eqref{eqn:lambda1} \text{ holds and } \mE(\bs\lambda)\neq 0\}.
\end{align*}
For each $\bs\lambda\in\bs\Lambda$ we set
\begin{align}\label{eqn:H-pointed}
\mH(\bs\lambda)&=\langle b_C- \lambda_C+\sum_{\mathclap{D\in\mR'(X,q)}}\lambda_DE_{CD}\mid C\in\mR' \rg.
\end{align}
We review the results from \S \ref{sec:liftings} in this setting:
\begin{pro}\label{pro:pointed-review}
Set $\mH=\widehat{\B}_2(X,q)\# H$. If $\bs\lambda\in\bs\Lambda$, then 
\begin{enumerate}\renewcommand{\theenumi}{\alph{enumi}}
\renewcommand{\labelenumi}{(\theenumi)}
\item $\mA(\bs\lambda)\coloneqq \mE(\bs\lambda)\# H\in\Cleft(\mH)$.
\item $\mH(\bs\lambda)$ is a cocycle deformation of $\mH$.
\item $\gr_{\mathfrak{F}} \mH(\bs\lambda) =\mH$. 
\end{enumerate}
If $\B(X,q)$ is quadratic, then
 \begin{enumerate}
 \item[(d)] $\mH(\bs\lambda)$ is a lifting of $V(X,q)$.\qed
 \end{enumerate}
\end{pro}

Recall the definition of the subalgebra $K\subseteq H$ as in \eqref{eqn:subalgebra}.

\begin{rem}
If we restrict to $K\subseteq H$, then \eqref{eqn:lambda1} becomes, 
\begin{align}\label{eqn:lambda1-K-1}
\lambda_C=q_{x,i_{h+1}}q_{x,i_{h}}\eta_h(C)  \lambda_D \quad \text{ if } D=x\rhd C[h].
\end{align}
\pf
Indeed, if $C=\{(i_2,i_1),\dots \}$, then by \eqref{eqn:alpha-HK}:
\begin{align*}
\lambda_C= \sum_{D\in\mR'}\delta_{D,x\rhd C[h]}\eta_h(C)q_{x,i_{h+1}}q_{x,i_{h}}  \lambda_D, \ x\in X.
\end{align*}
The claim follows since, for each $D\in\mR'$, there is at most one $x\in X$ and $h\in\I_{|C|}$, such that $D=x\rhd C[h]$. 
\epf\end{rem}

Set $\oG=\lg \ke_{xx}:x\in X\rg$ and let $\oK=\k\oG$ be the subquotient group algebra from Proposition \ref{pro:oK}. Let $(e_C)_{C\in\mR'}\in K$ be as in Remark \ref{rem:ee}.

\begin{pro}\label{pro:condition}
Let $(X,q)$ be such that $\B(X,q)$ is quadratic. Assume that  $X$ is faithful and that 
\begin{align}\label{eqn:condition}
\ke_C\neq \ke_{xx}, \qquad \forall\,x\in X, C\in\mR'.
\end{align}
Let $H$ be a Hopf algebra with realization $V=V(X,q)\in \ydh$. If $L$ is a lifting of $V$, then there is $\bs\lambda$ such that $L\simeq \mH(\bs\lambda)$.
\end{pro}
\pf
Following Lemma \ref{thm:liftings-step-3}, we need to show that $\homh(\mJ^2(V),V)=0$. Let $\varphi\in \homh(\mJ^2(V),V)$; set
\[
\varphi(b_C)=\sum_{x\in X}t_{C,x}v_x, \qquad (t_{C,x})_{x\in X}\in\k^X.
\]
Since $\varphi$ is $H$-colinear, we see that for every $h\in H$, $y\in X$ and $C\in\mR'$:
\begin{align*}
\sum_{x\in X} t_{C,x}e_{xy}&=\sum_{D\in\mR'}t_{D,y}E_{C,D}=\sum_{D\in\mR'}\sum_{h=1}^{|C|}t_{D,y}\frac{\eta_h(C)}{\eta_l(D)} e_{i_{h+1}j_{l+1}}e_{i_hj_l},
\end{align*}
using Lemma \ref{lem:structure-b}. This implies an equality in \eqref{eqn:condition}, a contradiction.
\epf

\begin{cor}\label{cor:condition}
Let $X$ be a finite indecomposable and faithful rack. Let $q$ be a 2-cocycle on $X$ and set $V=V(X,q)$. Let $H$ be a Hopf algebra with a realization $V\in\ydh$. Assume that either
\begin{enumerate}\renewcommand{\theenumi}{\alph{enumi}}
\renewcommand{\labelenumi}{(\theenumi)}
\item $q$ is constant; or
\item $X=\mO_2^4$, $q=\chi$ as in \eqref{eqn:MS}.
\end{enumerate}
If $L$ is a lifting of $V\in\ydh$, then there is $\bs\lambda$ such that $L\simeq \mH(\bs\lambda)$.
\end{cor}
\pf
We need to check \eqref{eqn:condition} in $\oG$. Notice that 
\[
\ke:X:\to \oG, \qquad x\mapsto g_x\coloneqq \ke_{xx},
\]
is injective, as $X$ is faithful. Then (a) is \cite[Lemma 3.7 (c)]{GV}. For (b), we proceed by inspection. First, assume that $g_x^2=g_y$. In particular, $x\neq y$ and $g_{x\rhd y}=g_xg_yg_x^{-1}=g_x^2=g_y$, so $x\rhd y=y$. This contradicts \cite[Lemma 3.7 (a)]{GV}, which establishes that in this case $g_x^2\neq g_y$.

Now, assume that $x\rhd y\neq y$. We may set, without lost of generality, that $x=(12)$, $y=(1a)$, $a=3,4$. Moreover, we can set $a=3$, so $z=(b4)$, $b=1,2,3$. Once again, we have that $\chi_z(g_z)=-1$ and $\chi_z(g_xg_y)=\chi_z(g_y)\chi_{y\rhd z}(g_x)=1$. This proves the claim.
\epf

\subsubsection{Principal realizations} If the realization is principal and $(g_C)_{C\in\mR'}\in G(H)$ is as in Remark \ref{rem:gC}, then 
\begin{align*}
\mH(\bs\lambda)&=T(V)\# H/\lg b_C - \lambda_C(1-g_C) : C\in \mR' \rg.
\end{align*}
If $G$ is a finite group and $H=\k G$, then $\mH(\bs\lambda)$ was introduced in \cite[Definition 3.6]{GG}, provided that \eqref{eqn:condition} holds.

In this case, a complete description of the isomorphism classes is achieved via standard arguments, see \cite[\S 6.]{GG}.
\begin{lem}\label{lem:isos}
Let $X$ be a faithful rack. Let $q$ be a 2-cocycle on $X$ such that $\B(X,q)$ is quadratic. 
Let $H$ be a cosemisimple Hopf algebra with a principal realization $V=V(X,q)\in\ydh$.
Then  $\mH(\bs\lambda)\simeq \mH(\bs\lambda')$ if and only if $\bs\lambda=\bs\lambda'=\bs0\in \k^{\mR'}$ or $[\bs\lambda]=[\bs\lambda']\in \P^{|\mR'|-1}$.\qed
\end{lem}

\subsection{Liftings associated to the symmetric groups}

We fix now $(X,q)$ as in \S \ref{sec:list}. That is, $X=\mO_2^4$ or $\mO_4^4$ and $q\equiv-1$ or $q=\chi$ as in \eqref{eqn:MS}.
We set $V=V(X,q)$ and let $H$ be a cosemisimple Hopf algebra with a realization $V\in\ydh$.

\begin{lem}\label{lem:nonzero}
If \eqref{eqn:lambda1} holds, then $\mE(\bs\lambda)\neq 0$. Hence
\[
\bs\Lambda=\{\bs\lambda=(\lambda_C)_{C\in \mR'} |\, \eps(h)\lambda_C=\sum_{D\in\mR'}\alpha_{C,D}(h)\lambda_D, \forall\, C\in\mR', h\in H\}.
\]
\end{lem}
\pf
This is a particular case of Proposition \ref{pro:nonzero}, 
as the algebras $\mE(\bs\lambda)$ are particular examples of the algebras considered in {\it loc.~cit.} We develop this in detail:
Recall that  the conjugacy classes $\{C\}_{C\in \mR'}$ that parametrize the space of relations of each Nichols algebra $\B(X,q)$ have either one, two or three elements. 
Now, \eqref{eqn:lambda1-K-1} implies that, if $|C|=|D|$, then $\lambda_C=\lambda_D$. We write $\lambda_i=\lambda_{|C|}$ if $|C|=i$, $i=1,2,3$. Moreover, when $q=\chi$, it follows that $\lambda_2=0$. Thus, $\mE(\bs\lambda)$ coincides with 
\begin{itemize}
\item $\mE_{\bs\alpha}(\lambda_2,\lambda_3)$ with $\alpha_{(ij)}=\lambda_1$, $(ij)\in\mO_2^n$; when $(X,q)=(\mO_2^4,-1)$;
\item $\mE^\chi_{\bs\alpha}(\lambda_3)$ with $\alpha_{(ij)}=\lambda_1$, $(ij)\in\mO_2^n$; when $(X,q)=(\mO_2^4,\chi)$;
\item $\widetilde{\mE}_{\bs\beta}(\lambda_1,\lambda_3)$ with $\beta_{\sigma}=\lambda_2$, $\sigma\in\mO_4^4$; when $(X,q)=(\mO_4^4,-1)$. 
\end{itemize}
This shows the claim.
\epf
\begin{exa}\label{exa:pointed}
Let us fix $(X,q)=(\mO_2^4,-1)$ and let $H$ be a cosemisimple Hopf algebra with a realization $V(X,q)\in\ydh$. Let $(e_{\sigma,\tau})_{\sigma,\tau\in X}\in H$ be as in \eqref{eqn:action-coaction}. We consider the subsets $X(2), X(3)\subset X\times X$:
\begin{align*}
X(2)&=\{((12),(34)), ((13),(24)),((23),(14))\},\\
X(3)&=\{((ij),(ik))\}_{i<j<k}\cup \{((ik),(ij))\}_{i<j<k}, 
\end{align*}
so $\mR'$ is in bijection with $X\times X(2)\times X(3)$, cf.~\S\ref{sec:examples}. Let $\bs\lambda\in\bs\Lambda$: recall that in this case $\bs\lambda=(\lambda_C)_{C\in\mR'}$ identifies with a triple $(\lambda_1,\lambda_2,\lambda_3)\in\bs\Lambda$, via $\lambda_C=\lambda_i$ if $|C|=i$. 

Thus $\mH(\bs\lambda)$ is the quotient of $T(V)\# H$ modulo the ideal generated by:
\begin{align*}
&a_\sigma^2=\lambda_1\left(1-\sum_{\mu\in X}e_{\sigma,\mu}^2\right)-\lambda_2\sum_{\mathclap{(\mu,\nu)\in X(2)}}e_{\sigma,\mu}e_{\sigma,\nu}-\lambda_3\sum_{\mathclap{(\mu,\nu)\in X(3)}}e_{\sigma,\mu}e_{\sigma,\nu};\\
&a_\sigma a_\tau+a_\tau a_\sigma=\lambda_2\left(1-\sum_{\mathclap{(\mu,\nu)\in X(2)}}e_{\sigma,\mu}e_{\tau,\nu}\right)-\lambda_1\sum_{\mu\in X}e_{\sigma,\mu}e_{\tau,\mu}-\lambda_3\sum_{\mathclap{(\mu,\nu)\in X(3)}}e_{\sigma,\mu}e_{\tau,\nu};\\
&a_\sigma a_\upsilon+a_{\sigma\upsilon\sigma} a_\sigma + a_\upsilon a_{\sigma\upsilon\sigma} =\lambda_3\left(1-\sum_{\mathclap{(\mu,\nu)\in X(3)}}e_{\sigma,\mu}e_{\upsilon,\nu}\right)-\lambda_1\sum_{\mu\in X}e_{\upsilon,\mu}e_{\upsilon,\mu}\\
&\hspace*{9.7cm} -\lambda_2\sum_{\mathclap{(\mu,\nu)\in X(2)}}e_{\sigma,\mu}e_{\upsilon,\nu};
\end{align*}
for all $ \sigma\in X$, $(\sigma,\tau)\in X(2)$ and $ (\sigma,\upsilon)\in X(3)$.
\end{exa}

\section{Liftings of $W(q,X)$}\label{sec:copointed}

We fix a rack and $2$-cocycle $(X,q)$ and $W(q,X)$ the associated braided vector space as in \eqref{eqn:rack-br-dual}.
In this section we study the liftings of $W(q,X)$.

\subsection{On the strategy applied to $\B(q,X)$}

We fix a cosemisimple Hopf algebra $H$ with realization $W=W(q,X)\in \ydh$.

Recall the definition of the subalgebra $K\subseteq H$ in \eqref{eqn:subalgebra}. 
We will assume throughout this section that $H=K$, i.e. that $H$ is generated by the comatrix elements attached to the realization $W\in\ydh$.
This is a technical assumption that is only present with the purpose of giving a more concrete, that is in terms of $(X,q)$, description of the liftings $\mH(\bs\lambda)$. Besides, this assumption is satisfied in the examples we target to study.

We present general results that will allow us to apply the strategy of \S \ref{sec:liftings} to find the liftings of $W\in \ydh$. As in \S\ref{sec:pointed}, we explore the first 
iteration of the strategy with $\Gc_0=\{\tilde b_C\}_{C\in \mR'}$ the set generators of the quadratic relations given in Corollary \ref{pro:quadratic-dual}. For each sequence of scalars 
$\bs\lambda=(\lambda_C)_{C\in\mR'}$ we set 
$$
\Gc(\bs\lambda)=\{\tilde b_C - \lambda_C : C\in \mR'\}\subset T(W).
$$
\begin{lem}\label{le:step 1 para copunteada}
$\ku\Gc(\bs\lambda)\subset T(W)$ is an $H$-submodule if and only if 
\begin{align}\label{eq:copointed condition-1}
\lambda_C&=0 &\text{if } \quad \delta_{x,y}&\neq \delta_{i_{2}\rhd(i_1\rhd x),y}\,q_{i_1,x}\,q_{i_{2},i_1\rhd x}
\end{align}
for some $x,y\in X$. Then the ($0$th) set of deforming parameters is
\begin{align}\label{eq:Lambda para BqX}
\bs\Lambda=\{(\lambda_C)_{C\in\mR'}|\mE(\bs\lambda)\neq 0 \text{ and \eqref{eq:copointed condition-1} holds} \}. 
\end{align}
\end{lem}
\begin{proof}
Let $\bs\Lambda'$ be the set of families $\bs\lambda=(\lambda_C)_C$ satisfying \eqref{eq:copointed condition-1}.
Let $(i_{h+1},i_h)$ in $C$ and $x,y\in X$. Then
\begin{align*}
e_{xy}\cdot(w_{i_h}w_{i_{h+1}})=\delta_{i_{h+1}\rhd(i_h\rhd x),y}\,q_{i_h,x}\,q_{i_{h+1},i_h\rhd x}\,w_{i_h}w_{i_{h+1}}.
\end{align*}
As $q$ is a 2-cocycle, we have that
\begin{align*}
q_{i_h,x}\,q_{i_{h+1},i_h\rhd x}&=q_{i_{h+1}\rhd i_h,i_{h+1}\rhd x}\,q_{i_{h+1},x}=q_{i_{h+2},i_{h+1}\rhd x}\,q_{i_{h+1},x}.
\end{align*}
On the other hand, notice that $i_{h+1}\rhd(i_h\rhd x)=i_{2}\rhd(i_1\rhd x)$.
Therefore 
\[e_{xy}\cdot(\tilde b_C-\lambda_C)=\delta_{i_{2}\rhd(i_1\rhd x),y}\,q_{i_1,x}\,q_{i_{2},i_1\rhd x}\,\tilde b_C-\delta_{x,y}\,\lambda_C.\]

Thus, it is clear that $\ku\Gc(\bs\lambda)$ is a $H$-module if $\bs\lambda\in\bs\Lambda'$. Instead, if $\bs\lambda\not\in\bs\Lambda'$ and $\ku\Gc(\bs\lambda)$ is a $H$-submodule, we see that $\tilde b_C$ and $1$ belong to $\ku\Gc(\bs\lambda)$ by acting by a suitable $e_{xy}$. However, as $\{\tilde b_C\}_{C\in\mR'}$ is linearly independent, $\tilde b_C, 1\notin\ku\Gc(\bs\lambda)$. By this contradiction we finish the proof.
\end{proof}

Next lemma concerns \texttt{Step 3}, recall \S \ref{sec:on step 3}.

\begin{lem}\label{le:step 3 para copunteada}
Let $f\in\homh(\mJ^2(W), W)$ and assume that $X$ is faithful. If $C=\{(i_2,i_1), \dots \}\in\mR'$, then $f(\tilde b_C)=\nu_C\, w_{j_C}$ for some $j_C\in X$ and $\nu_C\in\ku$. 
Moreover, if $\nu_C\neq0$, then
\begin{align}
\notag\delta_{j_C\rhd x,y}\,q_{j_C,x}&=\delta_{i_{2}\rhd(i_1\rhd x),y}\,q_{i_1,x}\,q_{i_{2},i_1\rhd x}\quad\forall x,y\in X\quad\mbox{and}\\
\label{eq:condiciones sobre f}
  e_{j_C,i}&=\begin{cases}
                \tilde E_{C,D},&\mbox{if $i=j_D$,}\\
                0,&\mbox{if $i\neq j_D$ for all $D\in\mR'$.}
               \end{cases}
  \end{align}
Reciprocally, every pair $({{\bf j},\bs\nu})$ of families $\bs\nu=(\nu_C)_{C\in\mR'}\in\k^{\mR'}$, ${\bf j}=(j_C)_{C\in\mR'}$, $j_C\in X$, satisfying \eqref{eq:condiciones sobre f} define a morphism $f=f_{{\bf j},\bs\nu}\in \homh(\mJ^2(W), W)$ by setting $f(\tilde b_C)=\nu_C\, w_{j_C}$.
\end{lem}

\begin{proof}
By assumption $f$ is a morphism of $H$-modules. Since $W$ is the direct sum of non-isomorphic simple $H$-modules, by Lemma \ref{prop:oK dual}, there are $j_C\in X$ and $\nu_C\in\ku$ for every $C\in\mR'$ such that $f(\tilde b_C)=\nu_C\, w_{j_C}$.

In the proof of Lemma \ref{le:step 1 para copunteada}, we saw that 
$$e_{xy}\cdot\tilde b_C=\delta_{i_{2}\rhd(i_1\rhd x),y}\,q_{i_1,x}\,q_{i_{2},i_1\rhd x}\,\tilde b_C.$$
Thus, the first equality of \eqref{eq:condiciones sobre f} holds because $f(e_{xy}\cdot\tilde b_C)=e_{xy}\cdot\, f(w_{j_C})$. Since $f$ also is a $H$-comodule map, the second equality of \eqref{eq:condiciones sobre f}  is immediate.

The reciprocal statement is clear.
\end{proof}

\begin{rem}\label{rem:condition for 0 copointed}
Let $C=\{(i_2,i_1), \dots \}\in\mR'$. The following are direct consequences of Lemmas \ref{le:step 1 para copunteada} and \ref{le:step 3 para copunteada}. 
\begin{enumerate}
\item If $\phi_{i_2}\phi_{i_1}\neq\id$ and $\bs\lambda\in\bs\Lambda$, then $\lambda_C=0$.
\item If $\phi_{i_2}\phi_{i_1}\neq\phi_j$ for all $j\in X$, then $f(\tilde b_C)=0$.
\end{enumerate}
\end{rem}

We can summarize the results from \S \ref{sec:liftings} in this context as follows.

\begin{pro}\label{pro:copointed-review}
Set $\mH=\widehat{\B}_2(q,X)\# H$. If $\bs\lambda\in\bs\Lambda$, then 
\begin{enumerate}\renewcommand{\theenumi}{\alph{enumi}}
\renewcommand{\labelenumi}{(\theenumi)}
\item $\mA(\bs\lambda)\coloneqq \mE(\bs\lambda)\# H\in\Cleft(\mH)$.
\item $\mH(\bs\lambda)$ is a cocycle deformation of $\mH$.
\item $\gr_{\mathfrak{F}} \mH(\bs\lambda) =\mH$. 
\end{enumerate}
If $\B(q,X)$ is quadratic, then
 \begin{enumerate}
 \item[(d)] $\mH(\bs\lambda)$ is a lifting of $W(q,X)$.
 \end{enumerate}
Assume also that $X$ is faithful and Remark \ref{rem:condition for 0 copointed} (2) holds, then 
 \begin{enumerate}
 \item[(d)] every lifting of $W(q,X)$ is isomorphic to $\mH(\bs\lambda)$ for some $\bs\lambda\in\bs\Lambda$.\qed
 \end{enumerate} 
\end{pro}

In the next subsections we will see that our main examples of braided vector spaces satisfy the hypothesis of the above proposition. 
We give a more concrete description of the algebras $\mH(\bs\lambda)$ for $H=\ku^{\Sn_4}$ in \S \ref{sec:cop-S4}, which allows us to study the isomorphism classes.

\subsubsection{The rack of transpositions}\label{sec:cop-transp}
Let $X=\cO_2^n$, $q\equiv-1$ and $\chi$ be the $2$-cocycle given in \S \ref{sec:examples}. Thus, $\mR'=\mR'(X,-1)=\mR'(X,\chi)$ is in bijection with $X\times X(2)\times X(3)$, 
recall Example \ref{exa:pointed}. We keep the notation above.

\begin{lem}\label{lem:cop-Lambda-1}
Let $W=W(X,-1)$ or $W(X,\chi)$. Then 
\begin{align}\label{eq:Lambda copunteado para X -1}
\bs\Lambda=\left\{(\lambda_C)_{C\in\mR'}| \, \lambda_C=0  \text{ if } C\neq\{ x,x\}, \forall\,x\in X\right\}
\end{align}
 and $\Hom_H^H(\mJ^2(W),W)=0$.
\end{lem}
\pf
In this setting, \eqref{eq:copointed condition-1}  states that $\lambda_C$ must vanish if $C\neq\{ x,x\}$. Hence  $\ku\Gc(\bs\lambda)$ is a $H$-submodule by Lemma 
\ref{le:step 1 para copunteada}. Moreover, $\mE(\bs\lambda)\neq0$ by Proposition \ref{pro:nonzero}, cf. Lemma \ref{lem:nonzero}. Indeed, $\mE(\bs\lambda)\simeq \mE_{\bs\alpha}(0,0)$ when 
$q\equiv-1$ and $\mE(\bs\lambda)\simeq \mE^\chi_{\bs\alpha}(0)$ when $q=\chi$, for $\alpha_{(ij)}=\lambda_C$, $C=\{(ij),(ij)\}\in\mR'$. This proves \eqref{eq:Lambda copunteado para X -1}. 

Finally, $\Hom_H^H(\mJ^2(W),W)=0$ by Lemma \ref{le:step 3 para copunteada} and Remark \ref{rem:condition for 0 copointed}.
\epf

\begin{rem}\label{rem:bijection-1}
In particular, $\bs\Lambda\simeq \k^X$, via $\lambda_C\mapsto \lambda_x$, if $C=\{(x,x)\}$.
\end{rem}

\subsubsection{The rack of $4$-cycles}\label{sec:cop-cycles}

Let $Y=\cO_2^4$ and $q\equiv-1$, recall \S \ref{sec:examples}. Then $\mR'=\mR'(Y,-1)$ consists of singletons $\bigl\{(\sigma,\sigma)\bigr\}$, pairs $\bigl\{(\sigma,\sigma^{-1}),\,(\sigma^{-1},\sigma)\bigr\}$ and clas\-ses of the form $\bigl\{(\sigma,\tau),(\nu,\sigma),(\tau,\nu)\bigr\}$ for $\sigma,\tau,\nu=\sigma\tau\sigma^{-1}\in Y$. The corresponding quadratic relations $\tilde b_C=b_C$ were listed in Theorem \ref{thm:nichols} (c).

\begin{lem}\label{lem:cop-Lambda-2}
Let $W=W(Y,-1)$. Then 
\begin{align}\label{eq:Lambda copunteado para Y -1}
\bs\Lambda=\left\{\bs\lambda\in\ku^{\mR'}| \, \lambda_C=0  \text{ if } C   \text{ is not a pair}\right\}
\end{align}
and $\Hom_H^H(\mJ^2(W),W)=0$.
\end{lem}
\pf
Follows as Lemma \ref{lem:cop-Lambda-1}, using Proposition \ref{pro:nonzero}.
\epf 

\begin{rem}\label{rem:bijection-2}
Here, $\bs\Lambda\simeq \k^Y$, via $\lambda_C\mapsto \lambda_\sigma$, if $C=\bigl\{(\sigma,\sigma^{-1}),\,(\sigma^{-1},\sigma)\bigr\}$.
\end{rem}

\subsection{Copointed Hopf algebras over $\Sn_4$}\label{sec:cop-S4}
In this section we specialize the results in \S \ref{sec:cop-cycles} and \ref{sec:cop-cycles} to the case $H=\k^{G}$, $G=\Sn_4$.
 
Let $X=\cO_2^4$ and $W\in\ydcd$ denote the principal realization of $W(-1,X)$, $W(\chi,X)$ or $W(-1,Y)$. 
The YD-structure on $W$ is given by:
\begin{align*}
\delta_t\cdot w_{x}&=\delta_{t,x}w_{x}, & \lambda(w_{x})&=\sum_{t\in G}\sgn(t)\delta_t\ot w_{t^{-1}xt}, & W&=W(-1,X);\\
\delta_t\cdot w_{x}&=\delta_{t,x}w_{x}, & \lambda(w_{x})&=\sum_{t\in G}\chi(t)\delta_t\ot w_{t^{-1}xt}, , & W&=W(\chi,X);\\
\delta_t\cdot w_{\sigma}&=\delta_{t,\sigma^{-1}}w_{\sigma}, & \lambda(w_{\sigma})&=\sum_{t\in G}\sgn(t)\delta_t\ot w_{t^{-1}\sigma t}, & W&=W(-1,Y);
\end{align*}
for each $x\in X$, $\sigma\in Y$ and $t\in\Sn_4$. 

\

The following lemma permits the application of our previous results.

\begin{lem}
Let $W=W(-1,X)$, $W(\chi,X)$ or $W(-1,Y)$. Then $\ku^{\Sn_4}$ is generated as an algebra by the comatrix elements associated to $W$.
\end{lem}

\begin{proof}
$W$ decomposes into the direct sum of three simple $\ku^{\Sn_4}$-comodules $U_i$ with $\dim U_i=i$ for $i=1,2,3$. This follows by dualyzing the $\Sn_4$-action. Then the coalgebra spanned by the comatrix elements associated to $W$ has dimension $14=1+4+9$. Therefore $K=\ku^{\Sn_4}$ because $\dim K$ has to divide $24=\dim\ku^{\Sn_4}$.
\end{proof}

\subsubsection{The liftings $\mH(\bs\lambda)$ up to isomorphism}
We give explicitly the presentation of the liftings of $W\in\ydcd$. Recall from Remarks \ref{rem:bijection-1} and \ref{rem:bijection-1}  that $\bs\Lambda\simeq \ku^X$,  $\bs\Lambda\simeq \ku^Y$, respectively. However, these sets are bigger than the family of isomorphism classes of liftings. 

For the rack $X$, we consider 
$$
\gA_4:=\biggl\{(\lambda_x)_{x\in X}\in\ku^{X}: \sum_{x\in X}\lambda_x=0\biggr\}.
$$
Here we write $\lambda_x\coloneqq \lambda_C$ when $C=\{(x,x)\}$. Given $\bs\lambda\in\gA_4$, we set
\begin{align*}
f_x^{\bs\lambda}=\sum_{g\in\Sn_4}(\lambda_x - \lambda_{g^{-1}xg})\,\delta_g \in \ku^{\Sn_4},\quad x\in X.
\end{align*}
The group $\Gamma_4:=\ku^{\times}\times\Aut(\Sn_4)$ acts on $\gA_4$ via $(\mu,\theta)\triangleright\bs\lambda=\mu(\lambda_{\theta(x)})_{x\in X}$. The class of $\bs\lambda$ in $\gA_4/\Gamma_4$ will be denoted $[\bs\lambda]$.

\begin{fed}\label{def: Aa}
Fix $\bs\lambda \in \gA_4$. We set
\begin{align*}
\cH_{[\bs\lambda]}\coloneqq T(-1, X)\#\ku^{\Sn_4}/\cI_{\bs\lambda}
\end{align*} 
where $\cI_{\bs\lambda}$ is the ideal generated by all the elements in \eqref{eq:Vn} and
\begin{align}\label{eq:rels-powers Aa}
w_x^2& - f_x^{\bs\lambda},\quad x\in X.
\end{align}
\end{fed}

\begin{fed}\label{def: Aachi}
Fix $\bs\lambda \in \gA_4$. We set
\begin{align*}
\cH^\chi_{[\bs\lambda]}\coloneqq (\chi,X)\#\ku^{\Sn_4}\hspace{-3pt}/\cJ_{\bs\lambda}
\end{align*}
where $\cJ_{\bs\lambda}$ is the ideal generated by all the elements in \eqref{eq:Wn} and \eqref{eq:rels-powers Aa}.
\end{fed}

\begin{pro}\label{prop:liftings of On2} 
\begin{enumerate}
\item If $L$ is a lifting of $W(-1,X)\in\ydcd$, then $L\simeq  \cH_{[\bs\lambda]}$ for one and only one $[\bs\lambda]\in\gA_4/\Gamma_4$.
\item If $L$ is a lifting of $W(\chi,X)\in\ydcd$, then $L\simeq  \cH^\chi_{[\bs\lambda]}$ for one and only one $[\bs\lambda]\in\gA_4/\Gamma_4$.
\end{enumerate}
\end{pro}

\begin{proof}
(1) Fix $W=(-1,X)$, the proof for $W(\chi, X)$ in (2) is identical. First, by Proposition \ref{pro:copointed-review} (e), $L\simeq\mH(\bs\lambda)$ for some $\bs\lambda\in\bs\Lambda$. Set 
$|\bs\lambda|=\sum_{x\in X}\lambda_x$ and $\bs\lambda'=\bs\lambda-|\bs\lambda|(1, \dots, 1)\in\gA_4$. Then it follows that $\mH(\bs\lambda)\simeq \cH_{[\bs\lambda']}$.

Assume now there is a Hopf algebra isomorphism $\Theta:\cH_{[\bs\lambda]}\rightarrow\cH_{[\bs\lambda]}$ for some $\bs\lambda,\bs\lambda'\in\gA_4$. Then we have  a group automorphism $\theta$ of $\Sn_4$ induced by $(\Theta_{|\ku^{\Sn_4}})^*$. Let $\phi_{\bs\lambda}$ and $\phi_{\bs\lambda'}$ be lifting maps for $\cH_{[\bs\lambda]}$ and $\cH_{[\bs\lambda']}$. Hence there are non-zero scalars  $c_{x}$, $x\in X$, such that 
\[
\Theta\phi_{\bs\lambda}(w_x)=c_x\phi_{\bs\lambda'}(w_{\theta(x)}), \qquad x\in X,
\]
by Lemma \cite[Lemma 5.6 (e)]{GV} and using the adjoint action of $\ku^{\Sn_4}$. Since $\Theta$ is a coalgebra map  we deduce $c_x=c$ for all $x\in X$, for some fixed $c\in\k$. Therefore $\bs\lambda'=(c^2,\theta)\triangleright\bs\lambda$ and hence $[\bs\lambda']=[\bs\lambda]$.
\end{proof}

We now consider the rack $Y=\mO_4^4$. Let
$$
\widetilde\gA:=\biggl\{(\lambda_{y})_{y\in Y}\in\ku^{Y}: \lambda_{y^{-1}}=\lambda_y\,\mbox{and}\, \sum_{y\in Y}\lambda_{y}=0\biggr\}.
$$
We consider the group $\Gamma_4$ acting on $\widetilde\gA$ via $(\mu,\theta)\triangleright\bs\lambda=\mu(\lambda_{\theta\sigma})_{\sigma\in\cO_4^4}$. As above $[\bs\lambda]$ denotes the class of $\bs\lambda$ in $\widetilde\gA/\Gamma_4$. Given $\bs\lambda\in\widetilde\gA$, we introduce
\begin{align*}
\fij{y}^{\bs\lambda}=\sum_{g\in\Sn_4}(\lambda_{y} - \lambda_{g^{-1}y g})\delta_g \in \ku^{\Sn_4},\quad y\in Y.
\end{align*}
Hence, $f_{y^{-1}}^{\bs\lambda}=f_y^{\bs\lambda}$ because $\lambda_{y^{-1}}=\lambda_y$ for all $y\in Y$.

\begin{fed}\label{def: Ca}
Fix $\bs\lambda \in \widetilde\gA$. We set 
\begin{align*}
\widetilde\cH_{[\bs\lambda]}\coloneqq T(-1,Y)\#\ku^{\Sn_4}/\cL_{\bs\lambda}
\end{align*}
where $\cL_{\bs\lambda}$ is the ideal generated by \eqref{eq:U} and
\begin{align}
\label{eq:rel in Ca}
x_y x_{y^{-1}}+x_{y^{-1}}x_y& - f_{y}^{\bs\lambda},\quad y\in Y.
\end{align}
\end{fed}

\begin{pro}\label{prop:liftings of Y} Let $L$ be a lifting of $W(-1,Y)\in\ydcd$. Then $L$ is isomorphic to $\widetilde\cH_{[\bs\lambda]}$ for one and only one $[\bs\lambda]\in\widetilde\gA_4/\Gamma_4$.
\end{pro}

\begin{proof}
It follows as Proposition \ref{prop:liftings of On2}.
\end{proof}

\section{Appendix}

In this section we introduce three families of algebras, all of which we show to be nontrivial using \texttt{GAP}, and that are essential to show that the Hopf algebras introduced along the article are indeed liftings of a given braided vector space. In particular, they are examples of PBW deformations of  Fomin-Kirillov algebras, as recently defined in \cite{HV}. 

Let us set, for $n\geq 3$, $X_n=\cO_2^n$ and $Y=\cO_4^4$ as in \S \ref{sec:examples}. Also, consider the constant cocycle $q\equiv-1$ on $X_n$ and $Y$, and the cocycle $\chi$ on $X_n$ as in \eqref{eqn:MS}.

\begin{fed}\label{def:algebras}
We set $V_n=W_n=\k X_n$, $U=\k Y$.
\begin{enumerate} [leftmargin=*]\renewcommand{\theenumi}{\alph{enumi}} 
\renewcommand{\labelenumi}{(\theenumi)}
\item For each family of scalars ${\bs\alpha}=(\alpha_{(ij)})_{i<j}\in \k$ and $\mu_1,\mu_2\in\k$, 
$\mE_{\bs\alpha}(\mu_1,\mu_2)$ is the quotient of $T(V_n)$ by the ideal generated by
\begin{align*}
\xij{ij}^2&=\alpha_{(ij)},\\
\xij{ij}\xij{kl}+\xij{kl}\xij{ij}&=\mu_1,\\
\xij{ij}\xij{ik}+\xij{ik}\xij{jk}+\xij{jk}\xij{ij}&=\mu_2,
\end{align*}
for all $(ij),(kl),(ik)\in X_n$ with $\#\{i,j,k,l\} = 4$. 

\smallskip
\item For each family of scalars ${\bs\alpha}=(\alpha_{(ij)})_{i<j}\in \k$ and $\mu\in\k$, 
$\mE^\chi_{\bs\alpha}(\mu)$ is the quotient of $T(W_n)$ by the ideal generated by
\begin{align*}
\xij{ij}^2&=\alpha_{(ij)},\\
\xij{ij}\xij{kl}-\xij{kl}\xij{ij}&=0 \,\mbox{ for all }\,\#\{i,j,k,l\} = 4, \\
\xij{ij}\xij{ik}-\xij{ik}\xij{jk}-\xij{jk}\xij{ij}&=\mu, \\
\xij{ik}\xij{ij}-\xij{jk}\xij{ik}-\xij{ij}\xij{jk}&=\mu,
\end{align*}
for all $1\leq i<j<k\leq n$. 

\smallskip
\item For each family of scalars ${\bs\beta}=(\beta_{\sigma})_{\sigma\in Y}\in \k$ and $\mu_1,\mu_2\in\k$, 
$\widetilde{\mE}_{\bs\beta}(\mu_1,\mu_2)$ is the quotient of $T(U)$ by the ideal generated by
\begin{align*}
x_{\sigma}^2&=\mu_1,\\
x_\sigma x_{\sigma^{-1}}+x_{\sigma^{-1}}x_\sigma&=\beta_{\sigma},\\
 x_\sigma x_{\tau}+x_\nu x_\sigma+x_\tau x_\nu&=\mu_2
\end{align*}
for all $\sigma,\tau\in Y$ with $\sigma\neq\tau^{\pm1}$ and $\nu=\sigma\tau\sigma^{-1}$.
\end{enumerate}
\end{fed}
\allowdisplaybreaks
\begin{pro}\label{pro:nonzero}
Assume $n=3,4$ and let $\mE$ be either $\mE_{\bs\alpha}(\mu_1,\mu_2)$, $\mE^\chi_{\bs\alpha}(\mu)$ or $\widetilde{\mE}_{\bs\beta}(\mu_1,\mu_2)$. Then $\mE\neq 0$.
\end{pro}
\pf
We check this using \texttt{GAP}, by computing a Gr\"obner basis for the ideal defining $\mE$. See files \texttt{O24-1.log}, \texttt{O24-chi.log} and \texttt{O44-1.log}. For instance, the  Gr\"obner basis for $\mE_{\bs\alpha}(\mu_1,\mu_2)$ 
is given by the generators of the ideal defining $\mE$ together with
\small
 \begin{align*}
&  x_{(13)}x_{(12)}x_{(13)} - x_{(12)}x_{(13)}x_{(12)} - \alpha_{(12)}+\alpha_{(13)}x_{(23)} - \mu_2x_{(13)} + \mu_2x_{(12)}, \\ &
  x_{(14)}x_{(12)}x_{(14)} - x_{(12)}x_{(14)}x_{(12)} - \alpha_{(12)}+\alpha_{(14)}x_{(24)} - \mu_2x_{(14)} + \mu_2x_{(12)}, \\ &
  x_{(14)}x_{(13)}x_{(12)} + x_{(14)}x_{(12)}x_{(23)} - x_{(23)}x_{(14)}x_{(13)} - \mu_2x_{(14)} + \mu_1x_{(13)}, \\ &
  x_{(14)}x_{(13)}x_{(23)} + x_{(14)}x_{(12)}x_{(13)} - x_{(23)}x_{(14)}x_{(12)} - \mu_2x_{(14)} + \mu_1x_{(12)}, \\ &
  x_{(14)}x_{(13)}x_{(14)} - x_{(13)}x_{(14)}x_{(13)} - \alpha_{(13)}+\alpha_{(14)}x_{(34)} - \mu_2x_{(14)} + \mu_2x_{(13)}, \\ &
  x_{(24)}x_{(23)}x_{(14)} - x_{(14)}x_{(12)}x_{(23)} - x_{(12)}x_{(24)}x_{(23)} - \mu_1x_{(24)} + \mu_2x_{(23)}, \\ &
  x_{(24)}x_{(23)}x_{(24)} - x_{(23)}x_{(24)}x_{(23)} - \alpha_{(23)}+\alpha_{(24)}x_{(34)} - \mu_2x_{(24)} + \mu_2x_{(23)}, \\ &
  x_{(14)}x_{(12)}x_{(13)}x_{(23)} - x_{(23)}x_{(14)}x_{(12)}x_{(23)} + \alpha_{(23)}x_{(14)}x_{(13)} + \mu_2x_{(23)}x_{(14)}\\
  &\qquad  + \mu_1x_{(12)}x_{(23)} -  \mu_1\mu_2, \\ &
  x_{(14)}x_{(12)}x_{(13)}x_{(14)} + x_{(13)}x_{(14)}x_{(12)}x_{(13)} + x_{(12)}x_{(13)}x_{(14)}x_{(12)} + \alpha_{(13)}\\
  &\qquad -\alpha_{(14)}x_{(24)}x_{(34)} + \mu_1 x_{(14)}x_{(13)} - \mu_2x_{(14)}x_{(12)} - \alpha_{(12)}+\alpha_{(13)}x_{(23)}x_{(24)} \\
    &\qquad - \mu_2x_{(13)}x_{(14)} + \mu_1x_{(13)}x_{(12)} + \mu_1x_{(12)}x_{( 14)} - \mu_2x_{(12)}x_{(13)} \\
      &\qquad - \alpha_{(13)}\mu_2-\mu_1\mu_2+\mu_2^2, \\ &
  x_{(14)}x_{(12)}x_{(23)}x_{(14)} + x_{(12)}x_{(14)}x_{(12)}x_{(23)} + \alpha_{(12)}-\alpha_{(14)}x_{(24)}x_{(23)} \\
    &\qquad - \mu_1x_{(14)}x_{(12)} - \mu_2x_{(23)}x_{(14)} - \mu_2x_{(12)}x_{(23)} + \mu_1\mu_2, \\ &
  x_{(14)}x_{(12)}x_{(13)}x_{(12)}x_{(23)} - x_{(23)}x_{(14)}x_{(12)}x_{(13)}x_{(12)} - \mu_2x_{(14)}x_{(12)}x_{(13)}\\
    &\qquad  - \mu_2x_{(23)}x_{(14)}x_{(1 3)} + \mu_2x_{(23)}x_{(14)}x_{(12)} + \mu_1x_{(12)}x_{(13)}x_{(12)} \\
      &\qquad + \alpha_{(12)}\alpha_{(13)}+\alpha_{(12)}\alpha_{(23)}-\alpha_{(13)}\alpha_{(23)}x_{(14)} + \mu_1\mu_2x_ {(13)} - \mu_1\mu_2x_{(12)}, \\ &
  x_{(14)}x_{(12)}x_{(13)}x_{(12)}x_{(14)}x_{(12)} + x_{(13)}x_{(14)}x_{(12)}x_{(13)}x_{(12)}x_{(14)} + \alpha_{(13)}\\
    &\qquad -\alpha_{(14)}x_{(14)}x_{(13)}x_{(24)}x_{ (34)} + \alpha_{(12)}-\alpha_{(13)}x_{(14)}x_{(12)}x_{(13)}x_{(24)} \\
      &\qquad - \mu_2x_{(14)}x_{(12)}x_{(13)}x_{(12)} - \alpha_{(12)}+\alpha_{(13)}x_{(23)}x_{(14)}x_{(12)}x_{( 24)} \\
        &\qquad - \mu_2x_{(13)}x_{(14)}x_{(12)}x_{(13)} - \mu_1x_{(13)}x_{(12)}x_{(14)}x_{(13)} - \mu_2x_{(13)}x_{(12)}x_{(14)}x_{(12)}\\
                  &\qquad  - \alpha_{(13)}+\alpha_{(14) }x_{(12)}x_{(14)}x_{(13)}x_{(34)} + \mu_1x_{(12)}x_{(14)}x_{(12)}x_{(23)} \\
    &\qquad - \mu_1x_{(12)}x_{(23)}x_{(14)}x_{(13)} - \mu_2x_{(12)}x_{(13)}x_{(14)} x_{(12)} - \mu_2x_{(12)}x_{(13)}x_{(12)}x_{(14)} \\
    &\qquad - \alpha_{(13)}\mu_2+\alpha_{(14)}\mu_2x_{(24)}x_{(34)} + \alpha_{(12)}\mu_1-\alpha_{(14)}\mu_1x_{(24)}x_{(23)} - \alpha_{(13)}\mu_1-\mu_1^2\\
    &\qquad  +\alpha_{(14)}\mu_1x_{(14)}x_{(34)} + \alpha_{(13)}\mu_2-\alpha_{(14)}\mu_2x_{(14)}x_{(24)} + \alpha_{(12)}\mu_1+\mu_2^2x_{(14)}x_{(12)} \\
    &\qquad +  \alpha_{(13)}\mu_1-\alpha_{(14)}\mu_1x_{(23)}x_{(34)} + \alpha_{(12)}\mu_2-\alpha_{(13)}\mu_2x_{(23)}x_{(24)} + \alpha_{(13)}\mu_2\\
    &\qquad  -\alpha_{(14)}\mu_2x_{(13)}x_{(34)}- \alpha_{(12)}\mu_2+\alpha_{(14)}\mu_2x_{(13)}x_{(24)} + \alpha_{(12)}\mu_1-\mu_1^2x_{(13)}x_{(14)}\\
    &\qquad   + \mu_2^2x_{(13)}x_{(12)}+ \alpha_{(12)}\mu_1-\alpha_{(13)}\mu_1x_{( 12)}x_{(24)} - \alpha_{(14)}\mu_2-\mu_1\mu_2\\
        &\qquad   +\mu_2^2x_{(12)}x_{(14)} + \alpha_{(12)}\alpha_{(14)}+\mu_2^2x_{(12)}x_{(13)} - \alpha_{(12)}\alpha_{(14)}\mu_2\\
                &\qquad   -\alpha_{(12)} \mu_1\mu_2+\alpha_{(13)}\mu_1\mu_2+\alpha_{(14)}\mu_2^2+\mu_1^2\mu_2-\mu_2^3, 
\\ &
  x_{(14)}x_{(12)}x_{(13)}x_{(12)}x_{(14)}x_{(13)} + x_{(12)}x_{(14)}x_{(12)}x_{(13)}x_{(12)}x_{(14)} - \alpha_{(12)}\\
    &\qquad +\alpha_{(14)}x_{(14)}x_{(12)}x_{(24)}x_ {(23)} - \alpha_{(13)}+\alpha_{(14)}x_{(14)}x_{(12)}x_{(23)}x_{(34)} \\
    &\qquad - \mu_2x_{(14)}x_{(12)}x_{(13)}x_{(12)} - \alpha_{(12)}+\alpha_{(14)}x_{(13)}x_{(14)}x_{(12)}x_ {(24)} \\
        &\qquad - \mu_1x_{(13)}x_{(14)}x_{(12)}x_{(13)} - \mu_2x_{(13)}x_{(12)}x_{(14)}x_{(13)}+ \mu_2x_{(12)}x_{(14)}x_{(12)}x_{(23)} \\
            &\qquad  - \mu_2x_{(12)}x_ {(14)}x_{(12)}x_{(13)} - \mu_2x_{(12)}x_{(23)}x_{(14)}x_{(13)} - \mu_1x_{(12)}x_{(13)}x_{(14)}x_{(12)}\\
                            &\qquad  - \mu_2x_{(12)}x_{(13)}x_{(12)}x_{(14)} +  -\alpha_{(13)}\mu_1+\alpha_{(14)}\mu_1x_{(24)}x_{(34)} + \alpha_{(12)}\mu_2
\\&\qquad -\alpha_{(14)}\mu_2x_{(24)}x_{(23)} - \alpha_{(12)}\mu_1+\alpha_{(14)}\mu_1x_{(14)}x_{(24)} +  \alpha_{(12)}\mu_1-\mu_1^2x_{(14)}x_{(13)} 
\\ &\qquad- \alpha_{(14)}\mu_2+\mu_2^2x_{(14)}x_{(12)} + \alpha_{(13)}\mu_2-\alpha_{(14)}\mu_2x_{(23)}x_{(34)} + \alpha_{(12)}\mu_1
\\ &\qquad -\alpha_{(13)}\mu_1x_{(23)}x_{(24)}+ \alpha_{(12)}\mu_2  -\alpha_{(14)}\mu_2x_{(13)}x_{(24)} + \alpha_{(12)}\alpha_{(14)}
\\ &\qquad-\mu_1^2+\mu_2^2x_{(13)}x_{(12)} + \alpha_{(12)}\mu_1- \mu_1^2x_{(12)}x_{(14)} + \mu_1\mu_2+\mu_2^2x_{(12)}x_{(13)}\\ &\qquad
 - \alpha_{(12)}\alpha_{(14)}\mu_2-\alpha_{(12)}\mu_1\mu_2+\alpha_{(13)}\mu_1\mu_2+\alpha_{(14)}\mu_2^2 +\mu_1^2\mu_2-\mu_2^3. 
\end{align*}
\epf

\end{document}